\documentclass[a4paper,12pt]{article}
\usepackage{amssymb,amsmath,amsthm}
\usepackage{marvosym} 
\usepackage{graphicx}

\usepackage{amsfonts}
\usepackage{latexsym} 
\usepackage{color} 
\usepackage[english]{babel}
\usepackage{hyperref}
\usepackage{enumitem} 
\usepackage{float}  
\numberwithin{equation}{section}

\newtheorem{lemma}{Lemma}[section]
\newtheorem{theorem}[lemma]{Theorem}
\newtheorem{proposition}[lemma]{Proposition}
\newtheorem{corollary}[lemma]{Corollary}
\newtheorem{remark}[lemma]{Remark}
\newtheorem{hypothesis}[lemma]{Hypothesis}

\newtheorem{definition}[lemma]{Definition}

\newcommand{\tg}{{\tt g}}

\newcommand{\R}{\mathbb R}
\newcommand{\cP}{\mathcal P}
\newcommand{\C}{\mathbb C}

\newcommand{\Z}{\mathbb Z}

\newcommand{\N}{\mathbb N}
\newcommand{\T}{\mathbb T}

\newcommand{\ii }{{\rm i} }

\def\norma#1{\|#1  \|}

\def\im{{\rm i}}

\def\sleq{\lesssim}

\def\be{{\bf e}}
\def\bji#1{{j_{#1}}}

\def\pil{\tilde P}
\def\norpil{ \|\pil\|_{R}}

\newcommand{\pa}{\partial}

\def\cZ{{\mathcal{Z}}}
\def\cJ{{\mathcal{J}}}

\definecolor{awesome}{rgb}{1.0, 0.13, 0.32}
\definecolor{darkgr}{rgb}{0.0, 0.62, 0.42}
\definecolor{cyan}{rgb}{0.0, 0.72, 0.92}

\def\bJ{{\bf J}}

\def\cM{\mathcal{M}}

\def\cG{{\mathcal{G}}}

\def\cV{{\mathcal{V}}}
\def\cK{{\mathcal{K}}}
\def\cT{{\mathcal{T}}}
\def\cI{{\mathcal{I}}}
\def\QS{{\mathcal{Q}\kern-0.3pt\mathcal{S}}}

\def\cE{{\mathcal{E}}}
\def\cM{{\mathcal{M}}}

\def\cU{{\mathcal{U}}}
\def\cR{{\mathcal{R}}}

\newcommand{\bral}{[ \! [} 
\newcommand{\brar}{] \! ]}

\title{Almost global existence for some Hamiltonian PDEs with small
  Cauchy data on general tori}

\author{
D. Bambusi\footnote{Dipartimento di Matematica, Universit\`a degli Studi di Milano, Via Saldini 50, I-20133
Milano. \newline
 \textit{Email: } \texttt{dario.bambusi@unimi.it}}, 
R.Feola\footnote{{Dipartimento di Matematica e Fisica, Universit\`a degli Studi RomaTre, Largo San Leonardo Murialdo 1, 00144 Roma}\newline
\textit{Email:} \texttt{roberto.feola@uniroma3.it}}, 
R.Montalto\footnote{Dipartimento di Matematica, Universit\`a degli Studi di Milano, Via Saldini 50, I-20133
Milano. \newline
 \textit{Email: } \texttt{riccardo.montalto@unimi.it}}
 }

\date{}

\begin{document}
  
\maketitle
%\listoftodos

\begin{abstract}
In this paper we prove a result of almost global existence for some
abstract nonlinear PDEs on flat tori and apply it to some concrete
equations, namely a nonlinear Schr\"odinger equation with a
convolution potential, a beam equation and a quantum hydrodinamical
equation. We also apply it to the stability of plane waves in NLS. The
main point is that the abstract result is based on  a nonresonance
condition much weaker than the usual ones, which rely on the celebrated Bourgain's Lemma which provides a partition of the ``resonant sites'' of the Laplace operator on irrational tori. 
\end{abstract}

\tableofcontents

\section{Introduction}
The problem of giving an upper bound on the growth of Sobolev norms in
Hamiltonian nonlinear PDEs, has been widely investigated. The results
typically obtained are known as ``almost global existence'': they
ensure that solutions corresponding to smooth and small initial data
remain smooth and small for times of order $\epsilon^{-r}$ with
arbitrary $r$; here $\epsilon$ is the norm of the initial datum.

We recall that there exist quite satisfactory results for semilinear
equations in one space dimension \cite{Bou96, Bam03,BG06}, which have also been
extended to some semilinear PDEs with unbounded perturbations
\cite{Yuan} and to some quasilinear wave equations \cite{Delort-2009},
gravity capillary water waves \cite{BD} (see also \cite{BeFeFran}),
quasi-linear Schr\"odinger \cite{FI} and pure gravity water waves
\cite{BeFePus} still in dimension one.  On the contrary for the case of higher dimensional
manifolds only particular examples are known
\cite{Delort-Imrekaz, BDGS,Delort-2015} and for PDEs in higher space dimension
with unbounded perturbations only partial results have been obtained
\cite{Ionescu-Pusateri, Feola-Greb-Ian, FeMo}. A slightly different point of view is the one developed in \cite{StaffWilson20}, \cite{HPSW21}, in which the authors analyze the phenomenon of energy transfer to high modes, for initial data Fourier supported in a box for the cubic NLS on the irrational square torus in dimension two. 

\noindent
To discuss the main
difficulty met in order to obtain almost global existence in more than
one space dimension, we recall that all the known results deal with
perturbations of linear systems whose eigenvalues are of the form $\pm
i\omega_j$ with $\omega_j$ real numbers playing the role of
frequencies. The main point is that in all known results, the
frequencies are assumed to verify a non-resonance condition of the
form
\begin{equation}
  \label{seconda}
\left|\omega_{j_1}\pm\omega_{j_2}\pm....\pm\omega_{j_r}\right|\geq
 (\max_{3}\left\{|j_1|,...,|j_r|\right\})^{-\tau} \gamma
\end{equation}
with $ \max_{3}\left\{|j_1|,...,|j_r|\right\} $ the third largest number
among $|j_1|,...,|j_r|$ and $\gamma,\tau$ positive numbers; condition
\eqref{seconda} is a kind of second Melnikov condition since it
requires to control linear combinations involving two frequencies with
index arbitrarily large.  \emph{This is a Diophantine type condition
  which is typically violated in more than one space dimensions}.

In the present paper we prove an abstract result of almost global existence
(see Theorem \ref{main.abs}) for some
Hamiltonian PDEs in which the linear frequencies are assumed to
fulfill the much weaker condition 
\begin{equation}
  \label{zero}
\left|\omega_{j_1}\pm\omega_{j_2}\pm....\pm\omega_{j_r}\right|\geq
(\max\left\{|j_1|,...,|j_r|\right\})^{-\tau}\gamma  
  \end{equation}
for all possible choices of indexes $j_1,...,j_r $. \emph{This is a
  condition typically fulfilled in any space dimension}. The key point
is that we also require the frequencies $\omega_j$ and the indexes $j$
to fulfill a structural property ensured by a Lemma by Bourgain on
the ``localization of resonant sites'' in $\T^d$. This
allows to prove a theorem ensuring that the Hamiltonian of the PDE can
be put in a suitable block-normal form which can be used to control
the growth of Sobolev norms. 

More precisely, following \cite{BG06}, we decompose the variables in
variables of large index (high modes) and variables of small index
(low modes); the normal form we construct is the standard Birkhoff
normal form for low modes, while it is a normal form in which the
equations of the high modes are linear time dependent
equations. Furthermore the equations for the high modes have a block
diagonal structure with dyadic blocks and this allows to control the
growth in time of the Sobolev norms.

We emphasize that one of the points of interest of our paper is that
it shows the impact of results of the kind of \cite{Bourgainlinear, DelortIMRN} dealing
with linear time dependent systems on nonlinear systems, thus, in view
of the generalizations \cite{risonante,unbounded,BL22}, it opens the
way to the possibility of proving almost global existence in more
general systems, e.g. on some manifolds with integrable geodesic flow.

\medskip

In the present paper, after proving the abstract result, we apply it to
a few concrete equations for which almost global existence was out of
reach with previous methods. Precisely we prove almost global
existence of small amplitude solutions (1) for nonlinear Schr\"odinger
equations with convolution potential, (2) for nonlinear beam equations and
(3) for a quantum hydrodinamical model (QHD) (see for instance \cite{FeIaMu}). 
We also prove Sobolev
stability 
of plane waves for the Schr\"odinger equation (following \cite{FGL13}).

To present in a more precise way the result, we recall that an
arbitrary torus can be easily identified with the standard torus
endowed by a flat metric. This is the point of view we will take. For
the Schr\"odinger equation we show that, without any restrictions on
the metric of the torus, one has that if the potential belongs to a
set of full measure then one has almost global existence. For the case
of the beam equation, we use the metric in order to tune the frequencies and to fulfill the
nonresonance condition, thus we prove that if the metric of the torus
is chosen in a set of full measure then almost global
existence holds. Examples of tori fulfilling our property are
rectangular tori with diophantine sides, but also more general tori
are allowed.

The result for the QHD model is very similar to that of the beam
equation: if the metric is chosen in
a set of full measure, then almost global existence holds. Also the
result of Sobolev stability of plane waves in the Schr\"odinger
equation is of the same kind: if the metric belongs to a set of full measure, one has stability of the plane
waves over times
longer than any inverse power of $\epsilon$. 

We recall that results
of this kind were only known for square tori in which the frequencies
have a structure identical to that of typical 1
dimensional systems. For irrational tori the only result ensuring
at least a quadratic lifespan of nonlinear Schr\"odinger equations with
unbounded, quadratic nonlinearities has been proved in \cite{FeMo}
(see also \cite{FeIaMu} for the Euler-Kortweg system and
\cite{BeFeGrIa} for the Beam equation).

\medskip Finally we recall the result \cite{BFG} in which the authors
consider a nonlinear wave equation on $\T^d$ and prove that if the
initial datum is small enough in some Sobolev norm then the solution
remains small in a weaker Sobolev norm for times of order
$\epsilon^{-r}$ with arbitrary $r$. The main difference is that this
result involves a loss of smoothness of the solution which is not
present in our result; however, we emphasize that at present our
method does no apply to the wave equation since no generalizations of
Bourgain's Lemma to systems of first order are known.

\noindent
{\sc Acknowledgments.} Dario Bambusi and Roberto Feola are supported by 
the research project PRIN 2020XBFL ``Hamiltonian and dispersive PDEs''.
%PRIN , 
%\ldots Roberto Feola is supported by PRIN \ldots 
Riccardo Montalto is supported by the ERC STARTING GRANT "Hamiltonian Dynamics, Normal Forms and Water Waves" (HamDyWWa), project number: 101039762.

\section{The abstract theorem}

\subsection{Phase space}\label{sezione spazio delle fasi}

Denote $\cZ^d:=\Z^d\times\left\{-1,1\right\}$. Let $g$ be a positive 
definite, {symmetric}, quadratic form on $\Z^d$ and, for
$J\equiv(j,\sigma)\in\cZ^d$, denote
\begin{equation}\label{modulo}
\begin{aligned}
& \left|J\right|^2\equiv|j|^2:= \sum_{i = 1}^d |j_i|^2 \, , \quad  |J|^2_g \equiv |j|_g^2 := g(j, j)\,. 
\end{aligned}
\end{equation}
We define
\begin{equation}\label{ell2s}
\begin{aligned}
\ell^2_s(\cZ^d;\C):=
&\Big\{
u\equiv(u_J)_{J\in\cZ^d}\ ,\quad
u_J\in\C\ ,\ :\ 
\\
&
\qquad\qquad\left\|u\right\|_s^2:=\sum_{J\in\cZ^d}\left(1+\left|J\right|\right)^{2s}|u_J|^2<\infty
\Big\}\ .
\end{aligned}
\end{equation}
In the following we will simply write $\ell^2_s$ for
$\ell^2_s(\cZ^d;\C)$ and $\ell^2$ for $\ell^2_0$. We denote by
$B_s(R)$ the open ball of radius $R$ and center $0$ in
$\ell^2_s$. Furthermore in the following $\cU_s\subset\ell^2_s$ will
always denote an open set containing the origin. 

We endow $\ell^2$ by the symplectic form $\im \sum_{j\in\Z^d} 
u_{(j,+)}\wedge u_{(j,-)}$, which, when restricted to $\ell_s^2$
($s>0$), is a weakly symplectic form.

Correspondingly, given a function $H\in C^1(\cU_s)$, for some $s$, its Hamilton equations are given by
\begin{equation}
  \label{Ham.eq}
\dot u_{(j,+)}=\im\frac{\partial H}{\partial u_{(j,-)}}\ ,\qquad \dot
u_{(j,-)}=-\im\frac{\partial H}{\partial u_{(j,+)}}\ ,
\end{equation}
or, compactly
\begin{equation}
\label{Ham.eq.1}
\dot u_{(j,\sigma)}=\sigma \im\frac{\partial H}{\partial u_{(j,-\sigma)}}\ .
\end{equation}
We will also denote by
\begin{equation}
  \label{campo}
X_H(u):=(X_J)_{J\in\cZ^d}\ ,\qquad X_{(j,\sigma)}:= \sigma
\im\frac{\partial H}{\partial u_{(j,-\sigma)}} 
\end{equation}
the corresponding (formal) Hamiltonian vector field.

In the following we will work on the space $\ell^2_s$ with $s$
large. More precisely, all the properties we will ask 
will be required to hold for all $s$ large enough.

\subsection{The class of functions (and perturbations)}\label{functions} 

Given an index $J\equiv(j,\sigma)\in\cZ^d$ we define the involution
\begin{equation}
  \label{Jbar}
\bar J:=(j,-\sigma)\, . 
\end{equation}
Given a multindex $\bJ\equiv (J_1,...,J_r)$, with $ J_l\in\cZ^d$,
$l=1,...,r$, we define $\bar\bJ:=(\bar J_1,...,\bar J_r)$.

\noindent
On the contrary, {\it for a complex number the bar will simply denote the
complex conjugate}.

\begin{definition}
  \label{invol}
  On $\ell^2_s$ we define the involution $I$ by
  \begin{equation}
    \label{invol.1}
(Iu)_{J}:=\overline{u_{\bar J}}\ .
  \end{equation}
  The sequences such that $Iu=u$ will be called \emph{real sequences}.
\end{definition}

\noindent
Given a multi-index $\bJ\equiv (J_1,...,J_{r})$, we also define its
momentum by
\begin{equation}
  \label{momentum}
\cM(\bJ):=\sum_{l=1}^{r}\sigma_lj_l\ .
\end{equation}
In particular in the following we will deal almost only with multi
indexes with zero momentum, so we define
\begin{equation}\label{indici}
\cI_r:=\left\{ \bJ\in(\cZ^d)^{r}\ :\ \cM(\bJ)=0  \right\}\, .
\end{equation}

Given a homogeneous polynomial $P$ of degree $r$, namely
$P:\ell^2_s\to\C$ for some $s$, it is well known that it can be
written in a unique way in the form 
\begin{equation}\label{forma}
P(u)=\sum_{J_1,...,J_{r}\in \cZ^d}P_{J_1,...,J_r}u_{J_1}...u_{J_r}\ ,
\end{equation}
with $P_{J_1,...,J_{r}}\in\C$ symmetric with respect to any permutation
of the indexes. 

We are now ready to specify the class of functions we will consider.

\begin{definition}{\bf (Polynomials).}\label{polyclass}
Let $r\geq1$. We denote by $\cP_r$ the space of 
formal polynomials $P(u)$ of the form \eqref{forma}
satisfying the following conditions:
%For $r\geq 1$, a homogeneous polynomial $P(u)=
%\sum_{\bJ}P_{\bJ}u_{J_1}...u_{J_{r}} $ of degree $r$ will be
%said to be of class $\cP_r$ if

\begin{itemize}
\item[P.1] \emph{(Momentum conservation):} $P(u)$ contains 
only monomyals with zero momentum, 
namely (recall \eqref{indici})  
\begin{equation}\label{momHam}
P(u)=\sum_{\bJ\in\cI_r}P_{\bJ}u_{J_1}...u_{J_r}\ ;
\end{equation}

\item [P.2] \emph{(Reality):} for any $\bJ\in (\mathcal{Z}^{d})^{r}$, 
one has $\overline{P_{\bar \bJ}}=P_{\bJ} $\,.
\item [P.3] \emph{(Boundedness):} The coefficients $P_{\bJ}$ are bounded, 
namely
  \[
  \sup_{\bJ\in\cI_r}|P_{\bJ}|<\infty\,.
  \] 
\end{itemize}
%In this case we will write $P\in\cP_r$. 

\noindent
For $R>0$
we endow the space 
$\cP_r$ 
with the family of norms 
\begin{equation}\label{normaPR}
\left\|P\right\|_R:=\sup_{\bJ\in\cI_r}|P_{\bJ}|R^{r}\,.
\end{equation}

\noindent
Given $r_2\geq r_1\geq1$
we denote by $\cP_{r_1,r_2}:=\bigcup_{l=r_1}^{r_2}\cP_l$ 
the space of polynomials $P(u)$ that may be written as
\[
P=\sum_{l=r_1}^{r_2}P_l\,,\qquad P_{l}\in \cP_{l}\,,
\]
endowed with the natural norm
\[
\left\|P\right\|_R:=\sum_{l=r_1}^{r_2}\left\| P_l\right\|_R\,.
\]
%and, if
%$P=\sum_{l=r_1}^{r_2}P_l$ with $P_l\in\cP_l$ we denote
%$\left\|P\right\|_R:=\sum_{l=r_1}^{r_2}\left\| P_l\right\|_R$. 
\end{definition}

\begin{remark}
  \label{r.classi}
  By the reality condition $(P.2)$ in Definition \ref{polyclass}, 
  one can note that if $P\in \cP_r$ then 
  
  \begin{itemize}
  \item $P(u)\in \R$ for all  real sequence $u$ (see Def. \ref{invol}). 
  %namely all $u$'s such that $u=Iu$;
  
  \item Fix $J_1,J_2\in \Z^d$  and define
\[
A_{J_1,J_2}(u):=\sum_{\substack{J_3,...,J_r \in \Z^d \\
(J_1,J_2,J_3,\ldots,J_r)\in \mathcal{I}_r}}P_{J_1,J_2,J_3,...,J_r}u_{J_3}...u_{J_r}\,.
\]
 Then, for all real sequence $u$, one has
    \begin{equation}
      \label{forsa}
A_{(j_1,+),(j_2,-)}=\bar A_{(j_2,+),(j_1,-)}\ ;
    \end{equation}
    this ``formal selfadjointness'' will play a fundamental
    role in the following.
  \end{itemize}
\end{remark}

\begin{definition}{\bf (Functions).}\label{funzP}
  We say that a function $P\in C^{\infty}(\cU_s;\C)$ belongs to class
$\cP$, and we write $P\in \mathcal{P}$, 
if 

\noindent
$\bullet$
all the terms of its Taylor expansion at $u=0$ are of class
$\cP_r$ for some $r$; 

\noindent
$\bullet$ the vector field $X_{P}$ (recall \eqref{campo}) 
belongs to $C^{\infty}(\cU_s;\ell^2_s)$
for all $s>d/2$.
%
%and furthermore $X_P\in
%C^{\infty}(\cU_s;\ell^2_s)$ for all $s>d/2$.  In this case we will
%write $P\in\cP$.
\end{definition}

The Hamiltonian systems that we will study are of the form
\begin{equation}
  \label{h.abs}
H=H_0+P\ ,
\end{equation}
with $P\in\cP$ and $H_0$  of the form
\begin{equation}
  \label{H0}
H_0(u):=\sum_{j\in\Z^d} \omega_ju_{(j,+)}u_{(j,-)}\ ,
\end{equation}
and $\omega_j\in\R$ a sequence on which we are going to make some
assumptions in the next subsection.

\subsection{Statement of the main result}\label{nonres}

We need the following assumption.

\begin{hypothesis}\label{hypo1}
The frequency vector $\omega=(\omega_{j})_{j\in \Z^d}$ satisfies the following.

\end{hypothesis}
\begin{itemize}
\item[F.1] There exist constants $C_1>0$ and $\beta>1$ such that,
  $\forall j$ large enough one has
 \[
\frac{1}{C_1}\left|j\right|^\beta\leq\omega_j\leq C_1
\left|j\right|^\beta\ .
\]
\item[F.2] For any $r\geq 3$ there exist $\gamma_r>0$ and $\tau_r$
  such that the following condition holds for all $N$ large enough
  \begin{align}
    \nonumber
\forall J_1,...,J_r\ \;\;\;\text{with}\ \;\;\; 
\left|J_l\right|\leq N\ ,\ \;\;\forall
l=1,...,r
\\
\label{prima}
\sum_{l=1}^r\sigma_{j_l}\omega_{j_l}\not=0\quad \Longrightarrow\quad 
\left|\sum_{l=1}^r\sigma_{j_l}\omega_{j_l}\right|\geq \frac{\gamma_r}{N^{\tau_r}}\,.
  \end{align}
\item[F.3] There exists a partition
  \begin{equation}
    \label{parti}
\Z^d=\bigcup_{\alpha}\Omega_\alpha\ ,
  \end{equation}
  with the following properties:
  \begin{itemize}
    
  \item[F.3.1]
    \begin{itemize}
    \item  either $\Omega_\alpha$ is finite dimensional and centered at the
      origin, namely there exists $C_1$ such that
      \[
      j\in\Omega_\alpha\ ,\quad \Longrightarrow \quad |j|\leq C_1\ ,
      \]
\item or it is dyadic, namely there exists a constant $C_2$ independent
    of $\alpha$ such that
    \begin{equation}
      \label{dya}
\sup_{j\in\Omega_\alpha}\left|j\right|\leq
C_2\inf_{j\in\Omega_\alpha}\left|j\right|\ .
    \end{equation}
    \end{itemize}
    \item[F.3.2] There exist $\delta>0$ and $C_3=C_3(\delta)$ such that, if
      $j\in\Omega_\alpha$ and $i\in\Omega_\beta$ with
      $\alpha\not=\beta$, then
      \begin{equation}
        \label{separa}
\left|i-j\right|+\left|\omega_i-\omega_j\right|\geq
C_3(\left|i\right|^\delta+\left|j\right|^\delta)\ .
      \end{equation}
  \end{itemize}
\end{itemize}

Finally, we need a separation property of the resonances, namely that
the resonances do not couple very low modes with very high modes. To
state this precisely, we first define an equivalence relation on
$\Z^d$

\begin{definition}
  \label{equi}
For $i,j\in\Z^d$, we say that $i\sim j$ if $\omega_i=\omega_j$. We
denote by $[i]$ the equivalence classes with respect to such an
equivalence relation.
  \end{definition}

\begin{hypothesis}\label{hypo2}
The frequency vector $\omega=(\omega_{j})_{j\in \Z^d}$ satisfies the following.

\begin{itemize}
\item[(NR.1)] The equivalence classes are dyadic, namely there exists
  $C>0$ such that
  \begin{equation}
    \label{dya.2}
C\inf_{j\in[i]}|j|\geq \sup_{j\in[i]}|j|\ ,\quad \forall i\in\Z^d\ ; 
    \end{equation}
\item[(NR.2)] \emph{Non-resonance}: $\forall l\in \N$ and   for any choice of
  $j_1,...,j_l$ such that $[j_i]\not=[j_k]$, if $i\not=k$, one has   
  \begin{equation}
    \label{nr.1}
\omega_{j_1}\pm\omega_{j_2}\pm...\pm \omega_{j_l}\not=0\,.
  \end{equation}
   \end{itemize}
   \end{hypothesis}
   \begin{remark}
   We point out that the Hypothesis \ref{hypo2} is only used in Section \ref{dynamics} in order to prove energy estimates for the system in normal form, see Lemma \ref{piccoli}.  
   \end{remark}
   
Our main abstract theorem pertains the Cauchy problem
\begin{equation}
  \label{Cauchy}
  \left\{\begin{aligned}
&   \dot u=X_{H}(u)
    \\
   & u(0)=u_0
  \end{aligned}
  \right.\ .
\end{equation}

\begin{theorem} \label{main.abs}
Consider the Cauchy problem \eqref{Cauchy} where $H$ has the form
\eqref{h.abs} with 
$H_0$ as in \eqref{H0}
and  $P\in \cP$
vanishing at order at least 3 at $u=0$.
Assume that the frequencies $\omega_j$ fulfill
Hypotheses \ref{hypo1}, \ref{hypo2}. 
%assumptions F.1,F.2,F.3 and also (NR.1), (NR.2). 
For any integer $r$
there exists $s_r\in\N$ such that
for any $s\geq s_r$  there exists $\epsilon_0>0$
and $c>0$ with the following property: if the initial datum
$u_0\in\ell^2_s$ is real and small, namely if
\begin{equation}
  \label{small}
  Iu_0=u_0\ ,\quad
  \epsilon:=\left\|u_0\right\|_{\ell^2_s}<\epsilon_0\ ,
\end{equation}
then the Cauchy problem \eqref{Cauchy} has a unique solution
$$
u\in C^0((-T_{\epsilon},T_{\epsilon}),\ell^2_s)\cap
C^1((-T_{\epsilon},T_{\epsilon}),\ell^2_{s-\beta})\ , 
$$
with $T_\epsilon>c\epsilon^{-r}$. Moreover there exists $ C>0$ such that
\begin{equation}
  \label{stima.solTeo}
\sup_{|t|\leq T_\epsilon}\left\| u(t)\right\|_{\ell^2_s}\leq
C\epsilon\ .
\end{equation}
\end{theorem}

The main step for the proof of Theorem \ref{main.abs} consists in
proving a suitable normal form lemma which is given in the next
section. 

\section{Normal form}\label{NF}

In the following we will use the notation $a\sleq b$ to mean there
exists a constant $C$, independent of all the relevant parameters,
such that $a\leq Cb$. If we want to emphasize the fact that the
constant $C$ depends on some parameters, say $r,s$, we will write
$a\sleq_{s,r}b$. 
%We also fix a number $\bar r$ related to the order of
%normalization that we will achieve. 
We will also write $a\simeq b$ if $a\sleq b$ and $b\sleq a$. 

We need the following definition.
\begin{definition}{\bf (N-block normal form).}\label{Nor.form}
Let $\bar{r}\geq3$ and $N\gg1$.
We say that a polynomial $Z\in\cP_{3,\bar r}$
of the form
\begin{equation*}
%\label{NF.1}
Z=\sum_{l=3}^{\bar r}\sum_{\bJ\in\cI_l}Z_{\bJ}u_{J_1}...u_{J_l}\ ,
\end{equation*}
(recall Def. \ref{polyclass})  
  %will be said to be
  is in \emph{$N$-block normal form} if
  $Z_{\bJ}\not=0$ \emph{only} if $\bJ\equiv(J_1,...,J_l)$ fulfills
  {\bf one} of the following two conditions:
  
  \begin{itemize}
  \item[1.] $|J_n|\leq N$ for any $n=1,\ldots,l$ and
    $\sum_{n=1}^l\sigma_{j_n}\omega_{j_n}=0$;

\item[2.] there exist \emph{exactly} $2$ indexes larger than $N$, say
  $J_1$ and $J_2$ and the following two conditions hold:
  \begin{itemize}
  \item[2.1] $J_1=(j_1,\sigma_1)$, $J_2=(j_2,\sigma_2)$ with
    $\sigma_1\sigma_2=-1$.

  \item[2.2] there exist $ \alpha $ such that $j_1,j_2\in\Omega_\alpha$, namely
    both the large indexes belong to the same cluster\footnote{recall conditions $F.3.1$, $F.3.2$ in 
  Hypothesis \ref{hypo1}} 
  $\Omega_\alpha$. 
  \end{itemize}
  
  \end{itemize}
\end{definition}

We now state the main result of this section.

\begin{theorem}\label{NF.theorem}
Fix any $N\gg1$, $s_0>d/2$ and 
consider the Hamiltonian \eqref{h.abs} with $\omega_j$ fulfilling
Hypothesis \ref{hypo1}
and $P\in\cP$. 
For any $\bar r\geq3$ 
there is $\tau>0$ such that for any
$s>s_0$ there exist $R_{s,\bar r}$, $C_{s,\bar r}>0$ 
such that for any $R<R_{s, \bar{r}}$ the following holds.
%Consider the Hamiltonian \eqref{h.abs} with $\omega_j$ fulfilling
%F.1,F.2,F.3 and $P\in\cP$. Fix $s_0>d/2$, for any $\bar r\geq3$ any
%$s>s_0$ there exists $R_{s,\bar r}$, $C_{s,\bar r}$ and $\tau, R_{\bar
%  r},$ (these  last
%paramenters independent of $s$) with the following properties: if
If
\begin{equation}
  \label{soglia}
RN^\tau<R_{s, \bar r} \ ,
\end{equation} then there exists  an invertible canonical transformation
\begin{equation}\label{stimamappaTT}
\cT^{(\bar r)}\,,\; [\cT^{(\bar r)}]^{-1}\,  :B_s(R)\to B_s(C_{s,\bar r}R)\ ,
%\quad \forall
%R<R_{s,\bar r}\ 
\end{equation}
%with the further property that 
%$$[\cT^{(\bar r)}]^{-1}:B_s(R)\to B_s(C_{s,\bar r}R)\ ,\quad \forall
%R<R_{s,\bar r}\  
%$$ and s.t.
such that 
\begin{equation}
  \label{Htra1}
H^{(\bar r)}:=H\circ\cT^{(\bar r)}=H_0+Z^{(\bar r)}+\cR_{T}+\cR_{\perp}
\end{equation}
where
\begin{itemize}
\item $Z^{(\bar r)}\in\cP_{3,\bar r}$ is in $N$-block normal form and fulfills
  \begin{equation}
    \label{zk1}
\left\|Z^{(\bar r)}\right\|_R\sleq_{\bar r}R^3\ ;
  \end{equation}
\item $\cR_{T}$ is such that $X_{\cR_{T}}\in
  C^{\infty}(B_s(R_{s,\bar r});\ell^2_s)$ and
  \begin{equation}
    \label{rtk.1}
\sup_{\left\|u\right\|_s\leq
  R}\left\|X_{\cR_{T,k}}(u)\right\|_s\sleq_{\bar r,s}R^2
(RN^\tau)^{\bar r-3}\ ,\quad \forall R\leq
R_{s,\bar r}\ ; 
    \end{equation}
\item $\cR_{\perp}$ is such that $X_{\cR_{\perp}}\in
  C^{\infty}(B_s(R_{s,\bar r});\ell^2_s)$ and
  \begin{equation}
    \label{rtk.2}
\sup_{\left\|u\right\|_s\leq
  R}\left\|X_{\cR_{\perp,k}}(u)\right\|_s\sleq_{
  \bar r,s}\frac{R^2}{N^{s-s_0}}\ ,\quad \forall R\leq
R_{s,\bar r}\ .
    \end{equation}
\end{itemize}
\end{theorem}

The rest of the section is devoted to the proof of this theorem and is
split in a few subsections. 

\subsection{Properties of the class of functions $\cP$}\label{properties}

First we give the following lemma.

\begin{lemma}{\bf (Estimates on the vector field).}
  \label{campo.1}
  Fix $r\geq 3$, $R>0$. Then for 
any $s > s_0 >d/2$  there exists a constant $C_{r,s}>0$
such that, $\forall P\in\cP_r$, the following inequality holds:
\begin{equation*}
  %\label{campo.2}
\left\|X_{P}(u)\right\|_s\leq C_{r,s}\frac{\left\|P\right\|_R}{R} \ ,\qquad \forall\, u\in B_s(R)\,.
\end{equation*}
\end{lemma}
\begin{proof}
  Let $P \in {\cal P}_r$. Then (recalling \eqref{campo}) one has 
  $X_P=((X_P)_J)_{J\in \cZ^d}$ with 
  \begin{equation}
    \label{expa.00}
(X_P)_{(j,+)}=\im \partial_{u_{(j,-)}} P= \im r  \sum_{\begin{subarray}{c}
J_1, \ldots, J_{r - 1} \in {\cal Z}^d, \, J=(j,+) \\
{\cal M}(J_1, \ldots, J_{r - 1}) + j = 0
\end{subarray}} P_{J,J_1,....,J_{r-1}} u_{J_1} \ldots u_{J_{r - 1}}
  \end{equation}
  and similarly for $(X_P)_{(j,-)}$. Remark that the r.h.s. of \eqref{expa.00} defines a unique
  symmetric $(r-1)$-linear form 
  \begin{equation}
    \label{expa.0}
(\widetilde{X_P})_{(j,+)}(u^{(1)},...,u^{(r-1)}):=\im r \!\! \sum_{\begin{subarray}{c}
J_1, \ldots, J_{r - 1} \in {\cal Z}^d \\
{\cal M}(J_1, \ldots, J_{r - 1}) +j = 0
\end{subarray}} P_{J,J_1,....,J_{r-1}} u_{J_1}^{(1)} \ldots u_{J_{r - 1}}^{(r-1)}\,.
    \end{equation}
In order to apply Lemma
  \ref{prod}  we decompose 
  \begin{equation}\label{expa.4}
u^{(l)}=u_+^{(l)}+u_-^{(l)}\ ,\ \text{with}\ 
  u_\sigma^{(l)}:=(u_{(j,\sigma)}^{(l)})_{j\in\Z^d}\ .
  \end{equation}
  Substituting in
  the previous expression we have 
  \begin{equation}\label{expa.1}
  \begin{aligned}
    (\widetilde{X_P})_{+}(u^{(1)}&,...,u^{(r-1)})=
    \\&=\sum_{l=0}^{r-1}
    \left(
    \begin{matrix}
r-1\\ l
    \end{matrix}
    \right) 
    (\widetilde{X_P})_{+}(u^{(1)}_+,...,u_+^{(l)},u_-^{(l+1)},...,u_-^{(r-1)} )\ .
    \end{aligned}
  \end{equation}
  Now each of the addenda of \eqref{expa.1} fulfills the assumptions
  of Lemma \ref{prod}. Therefore, since $\left\|u\right\|_{s_0}\leq\left\|u\right\|_{s}$ 
  one has 
\[
\begin{aligned}
&
\left\|
(\widetilde{X_P})_{+}(u^{(1)}_+,...,u_+^{(l)},u_-^{(l+1)},...,u_-^{(r-1)}
)\right\|_s
\\&\qquad\qquad\qquad\qquad\qquad
\sleq \sup_{\substack{(J,J_1,\ldots,J_r)\in (\mathcal{Z}^{d})^r}}\left|P_{J,J_1,...,J_{r-1}}\right|
\left\|u^{(1)}\right\|_s...\left\|u^{(r-1)}\right\|_s\ .
\end{aligned}
\]
Taking all the $u^{(l)}$ equal to $u\in B_{s}(R)$ (i.e. $\|u\|_s<R$) 
and recalling the norm in \eqref{normaPR}
one gets the thesis for
$({X_{P}})_+$. Similarly one gets the thesis for 
$({X_{P}})_-$ and this concludes the proof of the lemma. 
\end{proof}

As usual given two functions $f_1,f_2\in C^{\infty}(\ell^2_s;\C)$
we define their Poisson Brackets by
\begin{equation}
  \label{poisson}
\left\{f_1;f_2\right\}:=
\im\sum_{j\in\Z^d}\left(\frac{\partial
  f_1}{\partial u_{(j,-)}}\frac{\partial f_2}{\partial
  u_{(j,+)}}-\frac{\partial f_1}{\partial u_{(j,+)}}\frac{\partial
  f_2}{\partial u_{(j,-)}}\right)
  \equiv df_1 X_{f_2} \ ,
\end{equation}
which could be ill defined (but will turn out to be well defined in
the cases we will consider).

We recall that if both $f_1$ and $f_2$ have smooth vector field then
\begin{equation}
  \label{commu}
X_{\{f_1;f_2\}}=[X_{f_1};X_{f_2}]\ ,
\end{equation}
with $[\cdot\,;\,\cdot]$ denoting the commutator of vector fields.

\begin{lemma}{\bf (Poisson brackets).}
  \label{poi.1}
Given two polynomials $P_1\in\cP_{r_1}$ and $P_2\in\cP_{r_2}$, one has
$\left\{P_1;P_2\right\}\in\cP_{r_1+r_2-2}$ with
$$
\left\| \left\{P_1;P_2\right\} \right\| _{R}\leq
\frac{2r_1r_2}{R^2}\left\|P_{r_1}\right\| _{R}\left\|P_{r_2}\right\| _{R}\ .
$$
\end{lemma}
\begin{proof}
It follows by formula \eqref{poisson} 
recalling \eqref{normaPR} and exploiting the momentum conservation.
\end{proof}

We now fix some large $N>0$, but will track the dependence of all the
constants on $N$.
Corresponding to $N$ we define a decomposition of $u$ in \emph{low} and \emph{high}
modes. Precisely, we define the projectors
\begin{equation}\label{splitto1}
\Pi^{\leq}u:=(u_J)_{|J|\leq N}\ ,\quad \Pi^{\perp}u:=(u_J)_{|J|> N}\   
\end{equation}
and denote
\begin{equation}\label{splitto2}
u^{\leq}:=\Pi^{\leq}u\ ,\quad u^{\perp}:=\Pi^{\perp}u\ ,
\end{equation}
so that $u=u^{\leq}+u^\perp$.

As in \cite{Bam03,BG06}, a particular role is played by the
polynomials $P\in\cP_r$ which are quadratic or cubic in $u^\perp$. We
are now going to give a precise meaning to this formal
statement. First, given $f\in C^{\infty}(\cU_s;\C)$, we denote by
\[
d^lf(u)(h_1,...,h_l)\ 
\]
the $l$-th differential of $f$ evaluated at $u$ and applied to the
increments $h_1,...,h_l $.

\begin{definition}
  \label{zerok}
We say that $P\in\cP_r$ has a zero of order at least $k$ in
$u^{\perp}$ if
\[
d^lP(\Pi^{\leq}u)(\Pi^\perp h_1,...,\Pi^\perp h_l)=0\ ,\quad \forall\;
u,h_1,...,h_l\in \ell^2_s\ ,\;\;\; \forall\; l=0,...,k\ .
\]
We say that it is homogeneous of degree $k$ in $u^{\perp}$ if it has a zero
of order at least $k$, but not of order at least $k+1$.
\end{definition}

\begin{remark}
  \label{z0z2}
  By the very definition of normal form, one can decompose
  $Z^{(r)}=Z_0+Z_2$, with $Z_0$ homogeneous of degree zero in $u^\perp$
  and $Z_2$ homogeneous of degree 2 in $u^{\perp}$. Furthermore $Z_0$
  is in Birkhoff normal form in the classical sense, namely it
  contains only resonant monomyals, i.e. monomyals of the form
  \begin{equation*} %\label{resononat}
u_{J_1}...u_{J_r} \ ,  
\quad \text{with}\quad
\sum_{l}\sigma_l\omega_{j_l}=0\ .
 \end{equation*}
\end{remark}

\begin{lemma}
  \label{campo.N}
For all $s > s_0 >d/2$ and all $r\geq 3$, there exists a constant $C_{r,s}>0$
such that the following holds:
\begin{itemize}
\item[(i)] if $P\in\cP_r$ has a zero of order at least 2 in $u^\perp$,
  then
  \begin{equation*} %\label{campoN.1}
\sup_{\left\| u\right\|_s\leq R}\left\|\Pi^{\leq}X_{P}(u)\right\|_s
\leq
\frac{C_{r,s}}{N^{s- s_0}}\frac{\left\|P\right\|_R}{R}\ ;
 \end{equation*}
 
\item[(ii)] if $P\in\cP_r$ has a zero of order at least 3 in $u^\perp$,
  then
  \begin{equation*}
   % \label{campoN.2}
\sup_{\left\| u\right\|_s\leq R}\left\|X_{P}(u)\right\|_s\leq
\frac{C_{r,s}}{N^{s- s_0}}\frac{\left\|P\right\|_R}{R}\ .
  \end{equation*}
  \end{itemize}
\end{lemma}
\proof Consider first the case (i) and remark that, using the notation
\eqref{expa.4}, we have
$\left(\Pi^{\leq}X_P(u) \right)_{\pm}=\pm \im
\nabla_{u_{\pm}^{\leq}}P$, so that $\Pi^{\leq}X_P(u)$ has a zero of
order 2 in $u^{\perp}$. It follows that both in the case (i) and in the case
(ii)  we have to estimate a polynomial function $X(u)$ of the form 
\eqref{expa.00} with a zero of second order in $u^{\perp}$. To
exploit this fact consider first the $+$ component and consider again
the multilinear form $(\widetilde{X})_+$ as in \eqref{expa.0}: we have
\begin{equation*}
 % \label{dentrofuori}
X_+(u)=X_+(u^\perp+u^\leq)= \sum_{l=0}^{r-1}
    \left(
    \begin{matrix}
r-1\\ l
    \end{matrix}
    \right) 
    (\widetilde{X})_{+}(\underbrace{u^\perp,...,u^\perp}_{l-\text{times}},
    \underbrace{u^{\leq},...,u^{\leq}}_{r-1-l-\text{times}
      })\ ,
\end{equation*}
but, since $X_+(u)$ has a zero of at least second order in
$u^{\perp}$, one has
\begin{equation*}
 % \label{dentrofuori.1}
X_+(u)= \sum_{l=2}^{r-1}
    \left(
    \begin{matrix}
r-1\\ l
    \end{matrix}
    \right) 
    (\widetilde{X})_{+}(\underbrace{u^\perp,...,u^\perp}_{l-\text{times}},\underbrace{u^{\leq},...,u^{\leq}}_{r-1-l-\text{times}
      })\ .
\end{equation*}
Consider the first addendum (which is the one giving rise to worst
estimates): proceeding as in the proof of Lemma \ref{campo.1} one can
apply Lemma \ref{prod} and get the estimate
\begin{equation*}
\begin{aligned}
 &\left\|
 (\widetilde{X})_{+}(u^\perp,u^\perp,u^{\leq},...,u^{\leq})\right\|_s
 \sleq \sup_{j,j_1,\ldots, j_{r-1}\in\Z^d}\left|P_{j,j_1,---,j_{r-1}}\right| \times
 \\&\quad\quad\qquad\qquad\qquad\times
 \left(\left\|
 u^\perp\right\|_s\left\| u^\perp\right\|_{s_0}\left\|
 u^\leq\right\|^{r-3}_{s_0}+ \left\| u^\leq\right\|_s\left\|
 u^\perp\right\|^2_{s_0}\left\| u^\leq\right\|^{r-4}_{s_0}\right)\ ,
\end{aligned}
\end{equation*}
but 
$$
\left\| u^\perp\right\|_{s_0}^2=\sum_{|J|>N}\langle
J\rangle^{2s_0}\left|u_{J}\right|^2 =\sum_{|J|>N}\frac{\langle
  J\rangle^{2s}\left|u_J\right|^2}{\langle J\rangle^{2(s-s_0)}}\leq
\frac{\left\|u^\perp\right\|_s^2}{N^{2(s-s_0)}}=\frac{\left\|u\right\|_s^2}{N^{2(s-s_0)}}\ ,
$$
and thus
\[
\begin{aligned}
 \left\|
 (\widetilde{X})_{+}(u^\perp,u^\perp,u^{\leq},...,u^{\leq})\right\|_s
& \sleq 
\sup_{j,j_1,\ldots, j_{r-1}\in\Z^d}\left|P_{j,j_1,---,j_{r-1}}\right|
 \frac{R^{r-1}}{N^{s-s_0}}
 \\&
 \sleq \frac{\left\|P\right\|_R}{R}\frac{1}{N^{s-s_0}}\,.
\end{aligned}
\]
The other cases can be treated similarly.
\qed

\subsection{Lie Tranfsorm}\label{Lie}

Given $G\in\cP_{r,\bar r}$, consider its Hamilton equations $\dot u=X_G(u)$,
which, by Lemma \ref{campo.1}, are locally well posed in a
neighborhood of the origin.
Denote by $\Phi_G^t$ the corresponding flow, then we have the
following Lemma whose proof is equal to the finite dimensional case. 

\begin{lemma}
  \label{Lie.1}
  Consider $\bar{r}\geq r_1\geq r\geq3$ and $s>s_0>d/2$. 
There exists $C_{r,s}>0$ such that for any
$ G\in\cP_{r,r_1}$ 
and any $R>0$
satisfying 
  \begin{equation}
    \label{rs}
\frac{\left\|G\right\|_R}{R}\leq\frac{1}{C_{r,s}}\ ,
  \end{equation}
  the following holds.
For any $|t|\leq1$ one has $ \Phi_G^t(B_s(R))\subset B_s(2R)$
and the estimate
\begin{equation*} %\label{Lie.3}
    \sup_{u\in B_s(R)}\left\|\Phi^t_G(u)-u\right\|\leq \left|t\right|
    C_{r,s}\frac{\left\| G\right\|_R}{R} \ ,\quad  \forall t\ :\ |t|\leq
    1\ .
\end{equation*}
%  \begin{align}
%    \label{Lie.2}
%    \Phi_G^t(B_s(R))\subset B_s(2R)\ ,\quad \forall t\ :\ |t|\leq 1\ ;
%    \\
%    \label{Lie.3}
%    \sup_{u\in B_s(R)}\left\|\Phi^t_G(u)-u\right\|\leq \left|t\right|
%    C_{r,s}\frac{\left\| G\right\|_R}{R} \ ,\quad  \forall t\ :\ |t|\leq
%    1\ .
%  \end{align}
\end{lemma}

\begin{definition}
  \label{Lie.6}
The map $\Phi_G:=\Phi_G^t\Big|_{t=1}$ is called the \emph{Lie transform}
generated by $G$.
\end{definition}

In order to describe how a function is transformed under Lie transform
we define the operator
\begin{align*}
 % \label{Lie.7}
  Ad_G:C^{\infty}(\cU_s,\C)&\to C^{\infty}(\cU_s,\C)
  \\
  f&\mapsto Ad_G f:=\left\{f;G\right\}\ ,
\end{align*}
and its $k$-th power $Ad_G^kf:=\{Ad_{G}^{k-1}f;G\}$ for $k\geq1$. 
Also the following Lemma has a standard
proof equal to that of the finite dimensional case.
\begin{lemma}
  \label{Lie.8}
  Let $\bar{r}\geq r\geq3$ and $s>s_0>d/2$ and consider $G\in\cP_{r,\bar r}$. 
 There exists $C_{r,s}>0$ such that for any $R>0$ satisfying \eqref{rs}
 the following holds. For any  $f\in C^{\infty}(B_s(2R);\C)$
and any  $ n\in\N$ one has
\begin{equation}
  \label{Lie.9}
f(\Phi^t_G(u))=\sum_{k=0}^n\frac{t^k}{k!}(Ad_G^kf)(u)+\frac{1}{n!}\int_0^t(t-\tau)^n
(Ad_G^{n+1}f)\left(\Phi^\tau_G(u)\right) d\tau\ ,
\end{equation}
$\forall u\in B_{s}(R)$ and any $t$ with $|t|\leq 1$. 
\end{lemma}

From Lemma \ref{poi.1} one has the following corollary.

\begin{corollary}
  \label{cor.lie}
Let $G\in\cP_{r_1,r_2}$, $F\in\cP_{r_3,r_4}$, with
$r_1,r_2,r_3,r_4\leq \bar r$ and $3\leq r_1\leq r_2$. 
Let $\bar n\in \N$ be the smallest integer
such that $(\bar n+1)(r_1-2)+r_2>\bar r$. 
Then there exists $ C_{\bar r}>0$
such that for any $ k\leq \bar n$, one has
\begin{equation*}
  %\label{ad.k}
\left\|(Ad_G)^kF\right\|_R\leq \left(\frac{C_{\bar
    r}\left\|G\right\|_{R}}{R^2}\right)^k \left\| F\right\|_R\ .
\end{equation*}
\end{corollary}
A further standard Lemma we need is the following. 
\begin{lemma}
  \label{defo}
Let $G\in \cP_{r_1,r_2}$, $3\leq r_1,r_2\leq \bar r$ and let $\Phi_G$ be the
Lie transform it generates. Let $R_s$ by the largest value of $R$
such that \eqref{rs} holds. Then there exists $C>0$ such that for any $F\in
C^{\infty}(B_s(2R_s))$ satisfying 
\begin{equation*}
  %\label{orgin}
\sup_{\left\|u\right\|_s\leq
  2R}\left\|X_F(u)\right\|_s=:C_R<\infty\ ,\quad \forall R<R_s\ .
\end{equation*}
one has
\begin{equation*}
 % \label{orgin.1}
\sup_{\left\|u\right\|_s\leq
  R/C}\left\|X_{F\circ \Phi_G}(u)\right\|_s=:2C_R<\infty\ ,\quad \forall R<R_s\ .
\end{equation*}
\end{lemma}

From Lemma \ref{Lie.8}, Corollary \ref{cor.lie} and Lemma \ref{defo},
one has the following Corollary which is the one relevant for the
perturbative construction leading to the normal form lemma. 
\begin{corollary}
  \label{cor.3}
There exists $\mu_0>0$ such that for any
$G\in\cP_{r,\bar r}$, $3\leq r\leq \bar r$, the following holds. If 
\begin{equation*}
  %\label{mu}
\mu:=\frac{C_{\bar r}\left\| G\right\|_R}{R^2} <\mu_0\,,
\end{equation*}
with $C_{\bar r}$ the
constant of Corollary \ref{cor.lie},   then, for any
$F\in\cP_{r_1,\bar r}$, $r_1\leq \bar r$, one has
\begin{equation*}
  %\label{fcircG}
F\circ\Phi_G=\tilde F+\cR_{F,G}\ ,
\end{equation*}
with $\tilde F\in\cP_{r+r_1-2,\bar r}$ ($\tilde F\equiv 0$ if
$r+r_1-2>\bar r$) and $\cR_{F,G}\in C^{\infty}(B_s(R/C);\C)$ which
fulfill the following estimates
\begin{align*}
  %\label{circ.2}
  \left\|\tilde F\right\|_R\sleq_{\bar r} \mu\left\|F\right\|_R\ ,
  \\
  %\label{circ.3}
\sup_{\left\|u\right\|_s\leq R/C}
\left\|X_{\cR_{F,G}}(u)\right\|_{s}\sleq \frac{\left\|
  F\right\|_R}{R}\mu^{\bar n}\ ,
\end{align*}
with $\bar n$ as in Corollary \ref{cor.lie} and $C$ as in Lemma \ref{defo}.
\end{corollary}

\subsection{Homological equation}\label{homo.s}
In order to construct the transformation $\cT^{(\bar r)}$ of Theorem
\ref{NF.theorem}, we will use the Lie
transform generated by auxiliary Hamiltonian functions
$G_3,...,G_{\bar r}$, with
$G_\ell\in\cP_{\ell,\bar r}$, which in turn will be constructed by solving
the homological equation
\begin{equation}
  \label{homo.1}
\left\{H_0;G\right\}+Z=F
\end{equation}
with $F\in\cP_{\ell,\bar r}$ a given polynomial of order 2 in
$u^\perp$ and $Z$ to be determined, but in $N$-block normal form.  In
order to solve the homological equation we need a nonresonance
condition seemingly stronger than \eqref{prima}, but which actually
follows from $F.1$, $F.2$, $F.3$ of Hypothesis \ref{hypo1}.

To state the non-resonance condition we need the following definition.

\begin{definition}{\bf (Non resonant multi-indexes).}
  \label{non.res.index}
  For $\ell\in \N$ and $N\gg1$ we denote by 
$\cJ_l^N$ the subset of $\cI_l $ (see \eqref{indici}) of the  multi-indexes  
$\bJ=(J_1,...,J_l)$ satisfying the following conditions:
\begin{itemize}
\item[(I.1)] there are at most two indexes larger than $N$\,;

\item[(I.2)] if there exist two indexes, say $J_1$ and $J_2$ with
  $|J_1|>N$ and $|J_2|>N$, with
  $j_1$ and $j_2$ belonging to the same cluster\footnote{recall conditions $F.3.1$, $F.3.2$ in 
  Hypothesis \ref{hypo1}}
    $\Omega_\alpha$ then $\sigma_{1}\sigma_{2}=1$,
   \end{itemize}
Multi-indexes  ${\bf J}\in\cJ_l^N $ are called 
 \emph{non resonant multi-indexes}.
  \end{definition}
  
  \begin{remark}\label{rmk:normalform}
By Definitions \ref{Nor.form} and \ref{non.res.index}
  we notice that an Hamiltonian $Z\in \mathcal{P}_{r}$, $r\geq3$,
  of the form \eqref{momHam} but supported only on 
  multi-indexes ${\bf J}\in \mathcal{I}_{r}\setminus \mathcal{J}_{r}^{N}$
  is in $N$-block normal form.
  \end{remark}

\begin{lemma}
  \label{nr.vera}
  Assume Hypothesis \ref{hypo1} and let $r\in \N$.
%Assume that F.1, F.2, F.3 hold, then $\forall r\in\N$, $\exists$
Then there exist 
$\tau'_r$ and $\gamma'_r>0$, such that 
for any $3\leq p\leq r$ and any multi-index ${\bf J}\in\cJ_p^N$
one has
%for any 
%multiindex
%$\bJ\in\bigcup_{p=3}^r\cJ_p^N$ one has
\begin{equation}
  \label{NR.v}
\sum_{l=1}^{p}\sigma_l\omega_{j_l}\not=0\ \quad \Longrightarrow\ \quad 
\left|\sum_{l=1}^{p}\sigma_l\omega_{j_l}\right|\geq\frac{\gamma'_r}{N^{\tau'_r}}\ .
  \end{equation}
  \end{lemma}
\proof For the case where all the indexes $j_l$ are smaller than $N$ there
is nothing to prove.

Consider now the case where there is only one index, say $J_1$, larger
than $N$ and the length of the multi-index is $n+1\leq r$. The quantity
to be estimated is now
\begin{equation}
  \label{M.1}
\left|\sum_{l=2}^n\sigma_l\omega_{j_l}+\sigma_1\omega_{j_1}\right|\ .
  \end{equation}
By condition $F.1$, one has
$$
\left|\sum_{l=2}^n\sigma_l\omega_{j_l}\right|\leq rN^\beta C_1\ ,\quad 
{\rm and }\quad |\omega_{j_1}|\geq C_1|j_1|^\beta\,.
$$
%and $|\omega_{j_1}|\geq C_1|j_1|^\beta$, therefore,, if
Therefore, if 
$$
|j_1|\geq 2(rC_1^2)^{1/\beta}N=:N_1\ ,
$$
the estimate \eqref{NR.v} is satisfied. Hence, the estimate on the quantity \eqref{M.1} is nontrivial
only if all the indexes are smaller than $N_1$. 
It follows that we can
use \eqref{prima} with $N$ replaced by $N_1$, getting
$$
\left|\eqref{M.1}\right|\geq\frac{\gamma_r}{N_1^{\tau_r}}=\frac{\gamma_r}{2^{\tau_r}
(rC_1^2)^{\tau_r/\beta}N^{\tau_r}}\ ,
$$
which implies the bound \eqref{NR.v}   by choosing
$$
\gamma_r'\leq \frac{\gamma_r}{2^{\tau_r}
(rC_1^2)^{\tau_r/\beta}}\,.
$$

Consider now the case where there are two indexes larger than $N$, say
$J_1$ and $J_2$. The case $\sigma_1\sigma_2=1$ is dealt with similarly
to the previous case.

We discuss now to the case $\sigma_1\sigma_2=-1$. By condition $(I.2)$ 
in Definition \ref{non.res.index}
%of the definition of $\cJ^N_{n}$ 
there exist $\alpha\not=\beta$
such that $j_1\in\Omega_{\alpha}$ and $j_2\in\Omega_{\beta}$. It follows
that either
\begin{align}
  \label{B.1}
\left|\omega_{j_1}-\omega_{j_2}\right|\geq
C (\left|j_1\right|^{\delta}+\left|j_2 \right|^{\delta})
  \end{align}
or
\begin{align}
  \label{B.2}
\left|{j_1}-{j_2}\right|\geq
C(\left|j_1\right|^{\delta}+\left|j_2 \right|^{\delta})\ .
  \end{align}
for some $C  > 0$. Assume for concreteness  that $|j_1|\geq|j_2|$ and $\sigma_1=1$, $\sigma_2=-1$.

\noindent
Consider first the case where \eqref{B.1} holds. The quantity to be
estimated is
\begin{equation}
  \label{M.2}
\left|\sum_{l=3}^n\sigma_l\omega_{j_l}+\omega_{j_1}-\omega_{j_2}\right|\ .
\end{equation}
Notice that \eqref{B.1} implies $\left|\omega_{j_1}-\omega_{j_2}\right|\geq C |j_1|^\delta$
and that we also have
%$$
%\left|\omega_{j_1}-\omega_{j_2}\right|\geq C_3|j_1|^\delta\ ,
%\qquad \left|\sum_{l=3}^n\sigma_l\omega_{j_l}\right|\leq (r-2)N^\beta
%$$
%while 
$$
\left|\sum_{l=3}^n\sigma_l\omega_{j_l}\right|\leq (r-2)N^\beta\ .
$$
Then it follows that \eqref{NR.v} is automatic if
\[
\left|j_1\right|\geq\frac{2(r-2)^{1/\delta}}{C}N^{\beta/\delta}=:N_2\,.
 \]
Hence the bound on \eqref{M.2} is nontrivial only if all the indexes are smaller
than $N_2$. In this case   we can apply \eqref{prima} with $N_2$ in place of $N$,
getting
$$
\left|\eqref{M.2}\right|\geq\frac{\gamma_r}{N_2^{\tau_r}}=
\frac{\gamma_rC^{\tau_r}}{2^{\tau_r}(r-2)^{\tau_r/\delta}
  N^{\frac{\beta}{\delta}\tau_r} 
}\ ,
$$
which is the wanted estimate, in particular with
$\tau'_r\geq\frac{\beta}{\delta}\tau_r $.

\noindent 
It remains to bound  \eqref{M.2} from below  with indexes
fulfilling \eqref{B.2}.

\noindent
By  the zero momentum condition we have
$$
\sum_{l=3}^{n}\sigma_lj_l+j_1-j_2=0\ ,
$$
but $\left|\sum_{l=3}^{n}\sigma_lj_l  \right|\leq rN$, while
$$
\left|j_1-j_2\right|\geq C\left|j_1\right|^\delta\ .
$$
It follows that in our set there are no indexes with
$C |j_1|^\delta>rN$ (otherwise the zero momentum condition cannot be
fulfilled), so all the indexes must be smaller than
$N_3:=(rN/C)^{1/\delta}$, and again we can estimate $\eqref{M.2}$
using \eqref{prima} with $N$ substituted by $N_3$, thus getting the
thesis. 
\qed 

\begin{lemma}{\bf (Homological equation).}
  \label{sol.hom}
  Consider the Homological equation \eqref{homo.1}
  with $H_0$ as in \eqref{H0} and  $\omega_j$ satisfying Hypotheses
  \ref{hypo1} and where 
  $F\in\cP_{r,\bar r}$ is a polynomial having a zero of order 2 in
$u^\perp$. Then %the Homological 
equation \eqref{homo.1} has solutions
$Z\in\cP_{r,\bar r}$ and $G\in\cP_{r,\bar r}$ 
where $Z$ is in  $N$-block normal form, $N\gg1$
and moreover
\begin{align}
  \label{sol.h.1}
  \left\| Z\right\|_R&\leq 
  \left\|F\right\|_R\,,
  \\
  \label{sol.h.2}
  \left\|G\right\|_R&\leq
  \frac{N^{\tau'_{\bar r}}}{\gamma'_{\bar r}}\left\|F\right\|_R\ .
\end{align}
\end{lemma}

\proof Notice that, denoting $u_{\bJ}:=u_{J_1}...u_{J_r}$
and recalling \eqref{poisson},  one has
\begin{equation*}
\begin{aligned}
\left\{ H_0;u_{\bJ}\right\}&=
\im \sum_{j}\frac{\partial H_0}{\partial u_{(j,-)}}
\frac{\partial u_{\bJ}}{\partial u_{(j,+)}} -
\frac{\partial H_0}{\partial u_{(j,+)}}
\frac{\partial u_{\bJ}}{\partial u_{(j,-)}}
\\&
=\im \sum_{j}\omega_ju_{(j,+)}
\frac{\partial u_{\bJ}}{\partial u_{(j,+)}} -
\omega_ju_{(j,-)}
\frac{\partial u_{\bJ}}{\partial u_{(j,-)}}
\\&
=\im\sum_{j}\omega_ju_{\bJ}\left(\sum_{l=1}^{r}\delta_{(j,+),J_l}\omega_j-
\delta_{(j,-),J_l}\omega_j\right)
%=\im u_{\bJ}\left(\sum_{l=1}^{r}\delta_{(j,+),J_l}\omega_{j_l}-
%  \delta_{(j,-),J_l} \omega_{j_l} \right)
  \\&
  =\im u_{\bJ}\left(\sum_{l=1}^{r}\delta_{(j,+),(j_l,\sigma_l)}\omega_{j_l}-
    \delta_{(j,-),(j_l,\sigma_l)} \omega_{j_l} \right)
    =\im u_{\bJ}\sum_{l=1}^{r}\sigma_l\omega_{j_l}\ .
 \end{aligned}
 \end{equation*}
It follows that, writing
$$
P=\sum_{\bJ\in\cI_r}P_{\bJ}u_{\bJ}\ ,
$$
one can solve the Homological equation \eqref{homo.1} by defining (recall Def. \ref{non.res.index})
\begin{align*}
%  \label{solz}
  Z(u)&:=\sum_{\bJ\in\cI_r\setminus \cJ_r^N}P_{\bJ}u_{\bJ}\ ,
  \\
%  \label{solg}
  G(u)&:=\sum_{\bJ\in\cJ_r^N}\frac{P_{\bJ}}{\im\sum_{l=1}^{r}\sigma_l\omega_{j_l}}
  u_{\bJ}\ .
  \end{align*}
  By Remark \ref{rmk:normalform} we have that $Z$ is in $N$-block normal form.
 The estimates \eqref{sol.h.1}-\eqref{sol.h.2} immediately follow
using Lemma \ref{nr.vera}.
\qed

\subsection{Proof of the normal form Lemma}

Theorem \ref{NF.theorem} is an immediate consequence of the
forthcoming 
Lemma \ref{NF.Lemma}. To introduce it, we first split
$$
P=\pil+\cR_{T,0}\ ,
$$
with $\pil\in\cP_{3,\bar r}$ and $\cR_{T,0}$ having a zero of order at
least $\bar r+1$ at the origin. A relevant role will be played by the
quantity %$\|\tilde{P}\|_{R}$
$\norpil$. In order to simplify the notation, we remark that,
for $R$ sufficiently small there exists $K_{s,\bar r}$ such that
$$
\sup_{\left\|u\right\|_s\leq
  R}\left\|X_{\cR_{T,0}}(u)\right\|_s\leq K_{s,\bar r}\norpil R^{\bar
  r-3}\frac{1}{R} \ .
$$

\begin{lemma}{\bf (Iterative lemma).}
  \label{NF.Lemma}
  Assume Hypothesis \ref{hypo1} and 
  fix $\bar r\geq 3$. There exists $\mu_{\bar r}>0$ such that for any $3\leq k\leq\bar r$ and any
  $s>s_0>d/2$ there exist $R_{s,k}>0$, $C_{s,k},\tau>0$ 
  such that for any $R<R_{s,k}$ and any $N\gg1$ the following holds. If one has
\begin{equation}
  \label{mu.nf}
\mu:=\frac{\|\pil\|_R}{R^2}N^{\tau}<\mu_{\bar{r}}\,,
\end{equation}
 then there exists an invertible canonical transformation
\begin{equation}\label{TTk1}
\cT^{(k)}:B_s(R)\to B_s(C_{s,k}R)\ ,%\quad \forall R<R_{s,k}\ ,
\end{equation}
with
\begin{equation}\label{TTk2}
[\cT^{(k)}]^{-1}:B_s(R)\to B_s(C_{s,k}R)\ ,%\quad \forall R<R_{s,k}\ ,
\end{equation}
such that 
\begin{equation}
  \label{Htra}
H^{(k)}:=H\circ\cT^{(k)}=H_0+Z^{(k)}+P_k+\cR_{T,k}+\cR_{\perp,k}
  \end{equation}
  where
%and the following properties hold
\begin{itemize}
\item $Z^{(k)}\in\cP_{3,k}$ is in $N$-block normal form and fulfills
  \begin{equation}
    \label{zk}
\left\|Z^{(k)}\right\|_R\sleq_{\bar r,k}\norpil\ ;
  \end{equation}
\item $P_k\in\cP_{k,\bar r}$ fulfills
  \begin{equation}
    \label{pk}
\left\|P_k\right\|_R\sleq_{\bar r,k}\norpil \mu^{k-3}\,;
    \end{equation}
\item $\cR_{T,k}$ is such that $X_{\cR_{T,k}}\in
  C^{\infty}(B_s(R_{s,k});\ell^2_s)$ and
  \begin{equation}
    \label{rtk}
\sup_{\left\|u\right\|_s\leq
  R}\left\|X_{\cR_{T,k}}(u)\right\|_s\sleq_{\bar
  r,k,s}\norpil\mu^{\bar r-3}\frac{1}{R}\ ,\quad \forall R\leq
R_{s,k}\ ; 
    \end{equation}
\item $\cR_{\perp,k}$ is such that $X_{\cR_{\perp,k}}\in
  C^{\infty}(B_s(R_{s,k});\ell^2_s)$ and
  \begin{equation}
    \label{rtkperp}
\sup_{\left\|u\right\|_s\leq
  R}\left\|X_{\cR_{\perp,k}}(u)\right\|_s\sleq_{\bar
  r,k,s}\frac{\norpil}{R}\frac{1}{N^{s-s_0}}\ ,\quad \forall R\leq
R_{s,k}\ .
    \end{equation}
\end{itemize}
  \end{lemma}
The proof occupies the rest of the section and is split in a few
Lemmas. 
We reason inductively.
First, we consider the Taylor expansion of $P_k$ in $u^\perp$
and we write
\begin{equation}
  \label{pktay}
P_k=P_{k,eff}+R_{k,\perp}\ ,
\end{equation}
with $P_{k,eff}$ containing only terms of degree 0, 1 and 2 in
$u^\perp$, while $R_{k,\perp}$ has a zero of order at least 3 in
$u^\perp$. Then we determine $G_{k+1}$ and $Z_{k+1}$ by solving the
homological equation
\begin{equation}
  \label{homo.k}
\left\{H_0;G_{k+1}\right\}+P_{k,eff}=Z_{k+1}\,,
  \end{equation}
so that, by Lemma \ref{sol.hom} and the inductive assumption \eqref{pk}, we get
\begin{align}
  \label{stik1g}
  \left\|G_{k+1}\right\|_R&\sleq \left\|P_k\right\|_RN^{\tau}\sleq
  R^2\norpil \mu^{k-2}
  \\
  \left\|Z_{k+1}\right\|_{R}&\sleq \left\|P_k\right\|_R N^{\tau}\sleq
  \norpil \mu^{k-3}\,.
\end{align}
Consider the Lie transform $\Phi_{G_{k+1}}$ (recall Def. \ref{Lie.6}) generated by $G_{k+1}$.
By the estimate \eqref{stik1g} and  the condition \eqref{mu.nf}
we have that there is $R_{s,k+1}>0$ such that \eqref{rs} is fulfilled for $R<R_{s,k+1}$.
Hence Lemma \ref{Lie.1} applies and so we deduce that the map 
$\Phi_{G_{k+1}}$ is well-posed.
%
%Let $\Phi_{G_{k+1}}$ be the Lie transform generated by $G_{k+1}$ then
%(by \eqref{zero}, \eqref{mu.nf},\eqref{pk}), $\exists R_{s,k+1}$
%s.t. \eqref{rs} is fulfilled for $R<R_{s,k+1}$. 

We study now
$H^{(k)}\circ \Phi_{G_{k+1}}$. To start with we prove the following
Lemma.
\begin{lemma}
  \label{hzero}
  Let $G_{k+1}$ be the solution of \eqref{homo.k}, then one has
  \begin{equation}
    \label{gk1}
H_0\circ\Phi_{G_{k+1}}=H_0+Z_{k+1}-P_{k,eff}+\tilde
H_0+\cR_{H_0,G_{k+1}}\ ,
  \end{equation}
  with $\tilde H_{0}\in \cP_{k+1,\bar r}$, and, provided
  $R<R^0_{k+1}$, for some $R^0_{k+1}$, one has
  \begin{equation}
    \label{gk1.1}
\left\|\tilde H_0\right\|_R\sleq \mu^{k-2}\norpil \ .
  \end{equation}
  Furthermore, there exists $C_0>0$ such that one has
  \begin{equation}
    \label{gk1.2}
\sup_{\left\|u\right\|_s\leq R/C_0}\left\|X_{\cR_{H_0,G_{k+1}}
}(u)\right\|_s\sleq \mu^{\bar r-3}{\norpil}R^{-1}\ .
    \end{equation}
  \end{lemma}
\proof Let $\bar n$ be such that $(\bar n+1)(k-2)+k>\bar r$; using the
expansion \eqref{Lie.9} one gets
\begin{align}\nonumber
H_0\circ
\Phi_{G_{k+1}}=H_0+\left\{H_0;G_{k+1}\right\}+\sum_{l=2}^{\bar n}
\frac{Ad_{G_{k+1}}^l}{l!}H_0 +\cR_{H_0,G_{k+1}}
\\
\label{sera}
=H_0+\left\{H_0;G_{k+1}\right\}+\sum_{l=2}^{\bar n}
\frac{Ad_{G_{k+1}}^{l-1}}{l!}\left\{H_0;G_{k+1}\right\} +\cR_{H_0,G_{k+1}}
\end{align}
where we can rewrite explicitly the remainder term as
\begin{equation*}
 % \label{lis}
\cR_{H_0,G_{k+1}}=\frac{1}{\bar n!}\int_0^1(1-\tau)^{\bar
  n}\left(Ad_{G_{k+1}}^{\bar n}\left\{
H_0,G_{k+1}\right\}\right)\circ\Phi_{G_{k+1}}^{\tau}d\tau \ .
\end{equation*}
Since $G_{k+1}$ fulfills the Homological equation one has
$$
\left\{H_0,G_{k+1}\right\}=Z_{k+1}-P_{k,eff}\in\cP_{k,\bar r}\ ,
$$
with
$$
\left\|\left\{H_0,G_{k+1}\right\}\right\|_R\sleq \left\|P_k\right\|_R
\sleq \mu^{k-3}\norpil\ .
$$
Hence, defining $\tilde H_0$ to be the sum in Eq. \eqref{sera},
one has
\begin{equation*}
\begin{aligned}
 % \label{bella}
\left\|\tilde H_0\right\|_R&\equiv \left\|\sum_{l=2}^{\bar n}
\frac{Ad_{G_{k+1}}^{l-1}}{l!}\left\{H_0;G_{k+1}\right\}\right\|_R
\\
%\label{bellali}
&\sleq \sum_{l=2}^{\bar n}\left(\frac{C\left\|
  G_{k+1}\right\|_R}{R^2}\right)^{l-1} \frac{1}{l!}\mu^{k-3}\norpil
\sleq \mu^{k-3}\norpil \mu=\mu^{k-2}\norpil\ ,
\end{aligned}
\end{equation*}
provided $R$ is small enough. Analogously one gets
\begin{equation*}
 % \label{ds}
\sup_{\left\|u\right\|_s\leq
  R/C}\left\|X_{\cR_{H_0,G_{k+1}}}(u)\right\|_s\sleq
\mu^{k-3}\frac{\norpil\mu^{\bar n}}{R}\ ,
\end{equation*}
and, since $k+\bar n\geq \bar r$ the thesis follows.
\qed

In an analogous way one proves the following simpler Lemma whose proof
is omitted.

\begin{lemma}
  \label{stil}
Let $G_{k+1}\in\cP_{k,\bar r}$ fulfills the estimate \eqref{stik1g},
then we have
\begin{equation}\label{stip1}
\begin{aligned}
  %\label{stip1}
P_k\circ \Phi_{G_{k+1}}&=P_k+\tilde P_{k}+\cR_{P_k,G_{k+1}}\,,
\\
%\label{stiz1}
Z^{(k)}\circ \Phi_{G_{k+1}}&=Z^{(k)}+\tilde Z^{(k)}+\cR_{Z^{(k)},G_{k+1}}\,,
\end{aligned}
\end{equation}
and the following estimates hold
\begin{equation*}
\begin{aligned}
 % \label{stip2}
\left\|\tilde P_k\right\|_R&\sleq \norpil \mu^{k-2}\ ,\qquad \sup_{\left\|u\right\|_s\leq
  R/C}\left\|X_{\cR_{P_k,G_{k+1}}}(u)\right\|_s
  \sleq
\frac{\norpil\mu^{\bar r}}{R}\ ,
\\
\left\|\tilde Z^{(k)}\right\|_R&\sleq \norpil \mu^{k-2}\ ,\qquad \sup_{\left\|u\right\|_s\leq
  R/C}\left\|X_{\cR_{Z^{(k)},G_{k+1}}}(u)\right\|_s\sleq
\frac{\norpil\mu^{\bar r}}{R}\ .
\end{aligned}
\end{equation*}
\end{lemma}

\noindent
{\it End of the proof of Lemma \ref{NF.Lemma}.} 
We consider the Lie transform $\Phi_{G_{k+1}}$ generated by
$G_{k+1}$ determined by the equation \eqref{homo.k}
and we define
\[
\cT^{(k+1)}:=\cT^{(k)}\circ\Phi_{G_{k+1}}\,.
\]
By estimate \eqref{stik1g}, condition \eqref{mu.nf}, taking $R$ small enough, 
we have that Lemma \ref{Lie.1} applied to $G_{k+1}$
and the inductive hypothesis on $\mathcal{T}^{(k)}$ imply that 
$\mathcal{T}^{(k+1)}$ satisfies \eqref{TTk1}-\eqref{TTk2}
with $k\rightsquigarrow k+1$ and some constant $C_{s,k+1}$.

\noindent
Recalling \eqref{gk1}, \eqref{stip1} we define
\begin{align*}
  %\label{nuove}
Z^{(k+1)}&=Z^{(k)}+Z_{k+1}\ ,\quad P_{k+1}=\tilde P_k+\tilde
Z^{(k)}+\tilde H_0
\\
\cR_{T,k+1}&=\cR_{H_0,G_{k+1}}+\cR_{P_k,G_{k+1}}+\cR_{Z^{(k)},G_{k+1}}+\cR_{T,k}\circ
\Phi_{G_{k+1}}\ , \\
\cR_{\perp,k+1}&=R_{\perp,k}+ \cR_{\perp,k}\circ
\Phi_{G_{k+1}}\ .
\end{align*}
Then the iterative estimates follow from the estimates of Lemmas
\ref{hzero} and \ref{stil}. This concludes the proof.
\qed

\begin{proof}[{\bf Proof of Theorem \ref{NF.theorem}}]
Condition \eqref{soglia} implies \eqref{mu.nf}. 
Then the result follows by
Lemma \ref{NF.Lemma} taking $k=\bar{r}$.
\end{proof}

%We now take $N$ such that
%\begin{equation*}
% % \label{scelta.1}
%R
%%N^{\tau}\simeq R^{1/2}\ \iff \ N\simeq R^{1/2\tau}\ , 
%  \end{equation*}
%and $s_r$ s.t.
%\begin{equation}
%  \label{scelta.3}
%(RN^\tau)^{\bar r-3}\simeq \frac{1}{N^{s_r-s_0}}\ \iff \ R^{\frac{\bar
%    r-3}{2}}\simeq R^{\frac{s_r-s_0}{2\tau}}\ \Longleftarrow \ s_{\bar
%  r}=2(s_0+\tau(\bar r-3))\ .
%\end{equation}
%Finally choosing
%\begin{equation}
%  \label{scelta.2}
%\bar r=2r-1\ ,
%  \end{equation}
%we get the following Corollary of Theorem \ref{NF.theorem}.

An important consequence of 
Theorem \ref{NF.theorem} is the following.

\begin{corollary}\label{NF.corollary}
Consider the Hamiltonian \eqref{h.abs} with $\omega_j$ fulfilling
Hypotheses \ref{hypo1} 
and $P\in\cP$ (see Def. \ref{funzP}). 
For any $r\geq 3$ there exists $N_r>0$, $\tau>0$ and 
$s_r>d/2$ and a canonical
transformation $\cT_r$ such that
%with the following property: 
for any $s\geq
s_r$ there exists $R_s>0$ and $C_s>0$ such that the following holds  for any $R<R_{s}$:

\noindent
$(i)$
one has
\begin{align}
\cT_r&\in C^{\infty}(B_s(R/C_s);B_s(R))\ ,\quad \cT_r^{-1}\in C^{\infty}(B_s(R/C_s);B_s(R))\,,
\label{stimamappaT}
\\
H^{r}&:=H\circ \cT_r=H_0+Z^r+\cR^{(r)}\,,\label{hamrr}
\end{align}
%\begin{align}
%  \label{defo.c}
%\cT_r&\in C^{\infty}(B_s(R/C_s);B_s(R))\ ,\quad \cT_r^{-1}\in C^{\infty}(B_s(R/C_s);B_s(R))\ ,
%%\quad \forall R<R_s \ ,
%\\
%\label{c
%H^{r}&:=H\circ \cT_r=H_0+Z^r+\cR^{(r)}\ ,
%\end{align}
%and the following properties hold
where
\begin{itemize}
\item $Z^r\in\cP_{3,r}$ is in $N_r$-block normal form according to Def. \ref{Nor.form};
 \item $\cR^{(r)}$ is such that $X_{\cR^{(r)}}\in
  C^{\infty}(B_s(R_{s}/C_s);B_s(R_s))$ and
  \begin{equation}
    \label{rtk.1fin}
\sup_{\left\|u\right\|_s\leq
  R}\left\|X_{\cR^{(r)}}(u)\right\|_s\sleq_{r}R^{{r+1}}\ ,\quad \forall R\leq
R_{s}/C_{s}\ .
 \end{equation}
\end{itemize}

\vspace{0.5em}
\noindent
$(ii)$ Given $u\in B_s(R)$ we write $u=(u^{\leq}, u^{\perp})$
according to the splitting \eqref{splitto1}-\eqref{splitto2}
with $N$ replaced by $N_r$
and 
%also split the normal form $Z^{r}$ in Corollary \ref{NF.corollary} as
we set
$Z^r=Z_0+Z_2$ (see Remark \ref{z0z2}) 
where $Z_0$ is the part 
%according to Remark \ref{z0z2}, namely the part 
independent of
$z^{\perp}$ and $Z_{2}$ is the part homogeneous of order $2$ in $z^{\perp}$. 
Then we have
\begin{equation}\label{stimaZeta2}
\sup_{\left\|u\right\|_s\leq
  R}\left\|\Pi^{\leq}X_{Z_2}(u^{\leq}, u^{\perp})\right\|_s\sleq_{r}R^{r+1}\ ,
  \quad \forall R\leq
R_{s}/C_{s}\ .
\end{equation}
\end{corollary}

\begin{proof}
Let us fix
\begin{equation}
  \label{scelta.2}
\bar r=2r-1\ ,
  \end{equation}
 consider $\tau=\tau_r$ given by Lemma \ref{nr.vera} 
and fix
\[
s_{r}=2(s_0+\tau(\bar r-3))\,.
\]
We now take $N_r=N$ such that
\begin{equation}
  \label{scelta.1}
RN^{\tau}\simeq R^{1/2}\ \iff \ N\simeq R^{-1/2\tau}\ .
  \end{equation}
  With this choices the assumption \eqref{soglia}
  holds taking $R<R_s$ with $R_s$ small enough. Then
 Theorem \ref{NF.theorem} applies with 
  $s\geq s_r$, $N=N_r$ and $\tau=\tau_r$ chosen above.
  First of all notice that
  \begin{equation}
  \label{scelta.3}
(RN^\tau)^{\bar r-3}\simeq \frac{1}{N^{s_r-s_0}}\ \iff \ R^{\frac{\bar
    r-3}{2}}\simeq R^{\frac{s_r-s_0}{2\tau}}\,.
\end{equation}
Then formul\ae\,  \eqref{stimamappaT}-\eqref{hamrr} follow by \eqref{stimamappaTT}-\eqref{Htra1}
setting $\mathcal{R}^{(r)}=\mathcal{R}_{T}+\mathcal{R}_{\perp}$.
Then estimate \eqref{rtk.1fin} follows by \eqref{rtk.1}-\eqref{rtk.2} and \eqref{scelta.3}.
The estimate \eqref{stimaZeta2} follows by Lemma \ref{campo.N} 
and the choice of $N=N_r$ in \eqref{scelta.1}.
\end{proof}

\section{Dynamics and  proof of the main result} \label{dynamics}
In this section we conclude the proof of Theorem \ref{main.abs}. 

Consider the Cauchy problem \eqref{Cauchy}
 (with Hamiltonian $H$ as in \eqref{h.abs})
 with an initial datum $u_0$ satisfying \eqref{small}
 and fix any $r\geq3$.
 Recalling Hypotheses \ref{hypo1}, \ref{hypo2},
 setting 
 \begin{equation}\label{marconi1}
 \epsilon\simeq R\,,
 \end{equation}
  then
for $s\gg1$ large enough and $\epsilon$ small enough (depending on $r$),
 we  have that 
 the assumption of Corollary \ref{NF.corollary} are fulfilled.
 Therefore we set
 \[
 z_0:=\mathcal{T}_r(u_0)\,,
 \]
 and we consider the Cauchy problem
 \begin{equation}\label{Cauchy2}
 \dot{z}=X_{H^r}(z)\,,\qquad z(0)=z_0\,,
 \end{equation}
 with $H^{r}$ given in \eqref{hamrr}. By \eqref{stimamappaT}
 we have that the bound \eqref{stima.solTeo}
 on the solution $u(t)$ of \eqref{Cauchy} follows provided that
 we show
 \begin{equation}\label{claimR}
 \left\|z(t)\right\|_s\lesssim_s
\left\|z(0)\right\|_s+R^{{r+1}}|t|\ ,%\quad \forall t\ ,
\qquad \left|t\right|<T_R
%\sim R^{-r}\ ,
 \end{equation}
 where $z(t)$ is the solution of the problem \eqref{Cauchy2} and where
 we denoted 
 \begin{equation}\label{escapetime}
T_R:=\sup\left\{|t|\in\R^+\ :\ \left\|z(t)\right\|_s<R\right\}\,,
\end{equation}
the (possibly infinite) escape time of the solution from the ball of
radius $R$.
 
 \vspace{0.5em}
 \noindent
 The rest of the section is devoted to the proof of the claim \eqref{claimR}.
 To do this 
we now analyze the dynamics of the system \eqref{Cauchy2} obtained from the
normal form procedure. To this
end we write the Hamilton equations in the form of a system for the
two variables $(z^{\leq},z^{\perp})$ 
%(here $N\equiv N_r$ is given by Corollary \eqref{NF.corollary})
and also split the normal form $Z^r=Z_0+Z_2$ 
as in item $(ii)$ in Corollary \ref{NF.corollary}.
%
%according to Remark \ref{z0z2}, namely the part independent of
%$z^{\perp}$ and the part homogeneous of order $2$ in $z^{\perp}$. 
We get
\begin{align}
  \label{eq:NF.1}
\dot z^{\leq}&=\Lambda z^{\leq}+X_{Z_0}(z^{\leq})+\Pi^{\leq}
X_{Z_2}(z^{\leq}, z^{\perp})+\Pi^{\leq}
X_{\cR^{(r)}}(z^{\leq}, z^{\perp}) \ ,
\\
\label{eq:NF.2}
\dot z^{\perp}&=\Lambda z^{\perp}+\Pi^{\perp}X_{Z_2}(z^{\leq}, z^{\perp})+\Pi^{\perp}
X_{\cR^{(r)}}(z^{\leq}, z^{\perp})\ .
\end{align}
where $\Lambda$ is the linear operator such that $\Lambda z=X_{H_0}(z)$.
The key points to analyze the dynamics are 
the following:
\begin{itemize} 

\item[(i)] $Z_0$ is in
standard Birkhoff normal form;

\item[(ii)] by item $(i)$ of Lemma \ref{campo.N} one has that
$\Pi^{\leq} X_{Z_2}(z^{\leq}, z^{\perp})$ is a remainder term
(see item $(ii)$ in Corollary \ref{NF.corollary});

\item[(iii)] $\Pi^{\perp}X_{Z_2}(z^{\leq}, z^{\perp})$ is linear in
$z^{\perp}$. Furthermore, for any given trial solution $z^{\leq}(t)$
it is a time dependent family of linear operators, which by the
property \eqref{forsa} are selfadjoint and thus conserve the $L^2$
norm;

\item[(iv)] since $Z_2$ is in normal form it leaves invariant the dyadic
decomposition $\Omega_\alpha$ on which the $\ell^2$ norm is equivalent
to all the $\ell^2_s$ norms.
\end{itemize}
%
%that (1) $Z_0$ is in
%standard Birkhoff normal form; (2) by Lemma \ref{campo.N}, point (i),
%$\Pi^{\leq} X_{Z_2}(u^{\leq}, u^{\perp})$ is a remainder term, (3)
%$\Pi^{\perp}X_{Z_2}(u^{\leq}, u^{\perp})$ is linear in
%$u^{\perp}$. Furthermore, for any given trial solution $u^{\leq}(t)$
%it is a time dependent family of linear operators, which by the
%property \eqref{forsa} are selfadjoint and thus conserve the $L^2$
%norm; (4) since $Z_2$ is in normal form it leaves invariant the dyadic
%decomposition $\Omega_\alpha$ on which the $\ell^2$ norm is equivalent
%to all the $\ell^2_s$ norms.

Formally we split the analysis in a few lemmas. 
The first is
completely standard and provides a priori estimates on the low
frequency part
$z^{\leq}$ of the solution of \eqref{eq:NF.1}.

\begin{lemma}
  \label{piccoli}
There exists $K_1$ such that for any real initial datum $z_0\equiv
(z^{\leq}_0,z^{\perp}_0)$ for \eqref{eq:NF.1}, \eqref{eq:NF.2}, fulfilling
$\left\|z_0\right\|_s\leq R/2$ 
(with $R$ small as in \eqref{marconi1}) the following holds.
%with $R$ fulfilling \eqref{rtk.1}, 
%Let us denote by
%\begin{equation}\label{escapetime}
%T_R:=\sup\left\{|t|\in\R^+\ :\ \left\|z(t)\right\|_s<R\right\}\,,
%\end{equation}
%(the possibly infinite) escape time of the solution from the ball of
%radius $R$. 
One has that
\begin{equation}
  \label{bassi.1}
\left\|z^{\leq}(t)-z^{\leq}(0)\right\|_s\leq
K_1R^{r+1}\left|t\right|\ ,\quad \forall t\ ,\quad \left|t\right|<T_R\,,
  \end{equation}
  where $T_{R}$ is given in \eqref{escapetime}.
\end{lemma}
\proof For $i\in\Z^d$, define the ``superaction''
$$
J_{[i]}:=\sum_{j\in[i]}z_{(j,-)}z_{(j,+)}\equiv
\sum_{j\in[i]}\left| z_{(j,-)}\right|^2\ ,
$$
where the sum is over the indexes belonging to the equivalence class
of $[i]$ according to Definition \ref{equi} and the second equality
follows from the reality of $u$.
Then, by the property of being in normal form and by properties
(NR.1), (NR.2) in Hypothesis \ref{hypo2}, we have
$\left\{J_{[i]};Z_0\right\}=0
$, so that $\dot
J_{[i]}=\left\{J_{[i]};Z_2\right\}+\left\{J_{[i]};\cR^{(r)}\right\} 
  $ Denote by $\cE$ the set of all the equivalence classes of Definition
  \ref{equi}, and, for $e\in\cE$, denote
  \begin{equation*}
    %\label{equi.44}
\langle e\rangle:=\inf_{i\in e}\langle i\rangle
    \end{equation*}
  then the norm
  \begin{equation*}
   % \label{equi.n}
\left|z\right|_s^2:=\sum_{e\in\cE}\langle e\rangle^{2s} J_{e}
  \end{equation*}
  is equivalent to the standard one on $\Pi^{\leq}\ell^2_s$.  Thus we
  have  
  \begin{align*}
\frac{d}{dt}\left|z^{\leq}\right|_s^2=\sum_{e\in\cE}\langle e\rangle^{2s}
\frac{d}{dt} J_{e} =\sum_{e\in\cE}\langle e\rangle^{2s}
\left(\left\{J_{e};Z_2\right\} +\left\{J_{e};\cR^{(r)}\right\} \right)
\\
= d\left( \left|z^{\leq}\right|_s^2\right)
\left(X_{Z_2}
(z^{\leq},z^{\perp})+X_{\cR^{(r)}}
(z^{\leq},z^{\perp})\right)\ . 
\end{align*}
Then by \eqref{rtk.1fin}-\eqref{stimaZeta2}
 %by Lemma \ref{campo.N}, eq. \eqref{campoN.2} and \eqref{rtk.1fin},
%\eqref{brtk1},
the last quantity is estimated by a constant times $R^{r+2}$.
%$$
%2R\frac{R^2}{N^{s-s_0}}+R^{r+2}\sleq R^{r+2}\ .
%$$ where we also used the choices \eqref{scelta.1}, \eqref{scelta.3},
%\eqref{scelta.2}; 
From this, denoting by $K_0$ the constant in the
above inequality, one gets
$$
\left|z^{\leq}(t)-z^{\leq}(0)\right|_s\leq
K_0R^{r+1}\left|t\right|\ . 
$$
So we have 
$$
\left\|z^{\leq}(t)-z^{\leq}(0)\right\|_s\leq C
\left|z^{\leq}(t)-z^{\leq}(0)\right|_s \leq
CK_0R^{r+1}\left|t\right|
$$
from which, writing $K_1:=K_0C$ one gets the estimate \eqref{bassi.1}. 
\qed

We now provide a priori estimates on the high  frequencies
$z^{\perp}$ which evolve according to \eqref{eq:NF.2}.
\begin{lemma}
  \label{grandi}
For any $s>s_0$ there exists $K_2=K_2(s)$ such that for any real initial datum $z_0\equiv
(z^{\leq}_0,z^{\perp}_0)$ for \eqref{eq:NF.1}, \eqref{eq:NF.2}, fulfilling
$\left\|z_0\right\|_s\leq R/2$ (with $R$ small as in \eqref{marconi1}) %fulfilling \eqref{rtk.1},
the following holds.
%Denote as in \eqref{escapetime} by 
%$T_R$
%(the possibly infinite) escape time of the solution from the ball of
%radius $R$; then we have
One has that
\begin{equation}
  \label{grandi.1}
\left\|z^{\perp}(t)\right\|_s\leq
K_2\left\|z^{\perp}(0)\right\|+K_3R^{r+1}|t|\ ,\quad \forall t\ ,\quad \left|t\right|<T_R\,,
  \end{equation}
  where $T_{R}$ is given by \eqref{escapetime}.
\end{lemma}
\proof First, we denote by $\cZ(z^{\leq}):\Pi^\perp\ell^2\to
\Pi^\perp\ell^2$ the family of linear operator
s.t. $X_{Z_2}(z^{\leq},z^{\perp})=\cZ(z^{\leq}) z^{\perp}$; We also
write $\cZ(t):=\cZ(z^{\leq}(t))$, with $z^{\leq}(t)$ the projection on
low modes of the considered solution. We now introduce some further
notations. For any $z \in \Pi^{\perp}\ell^2$, we introduce the
projector $\Pi_\alpha$ associated to the block $\Omega_\alpha$ of the
partition. More precisely, for any $\alpha $, we define
\begin{equation}\label{proiettore Pi alpha}
\Pi_\alpha : \Pi^{\perp}\ell^2 \to \Pi^{\perp}\ell^2, \quad \Pi_\alpha
u:=
\left\{
\begin{matrix}
  z_{(j,\sigma)} &\text{if}&j\in\Omega_\alpha
  \\
  0&\text{if}&j\not\in\Omega_\alpha
\end{matrix}\right.\ .
\end{equation}
Then any sequence $z \in \Pi^{\perp}\ell^2 $ can be written as 
\begin{equation}\label{spectral decomposition Bourgain-Delort}
z = \sum_{\alpha} z_\alpha, \quad z_\alpha := \Pi_\alpha u \ . 
\end{equation}
By the property $2.2$ of Definition \ref{Nor.form}, 
the normal form operator ${\cal Z}(t)$ has a block-diagonal structure, namely it can be written as 
\begin{equation}\label{decomposizione cal Z (t)}
{\cal Z}(t) = \sum_{\alpha} {\cal Z}_\alpha(t), \quad {\cal
  Z}_\alpha(t) := \Pi_\alpha {\cal Z}(t) \Pi_\alpha\  . 
\end{equation}
For any block $\Omega_\alpha$, we define 
\begin{equation*}%\label{n (alpha)}
n(\alpha) := {\rm min}_{j \in \Omega_\alpha} \langle j \rangle 
\end{equation*}
and for any $z \in \ell^2_s$, we define the norm 
%(warning: some
%notations are not coherent with the notations in the proof of Lemma
%\ref{piccoli}. In particular the norm denoted by $|.|_s$ is
%different).  
\begin{equation*}%\label{def norma equiv Hs}
\bral z\brar_s := \Big( \sum_{\alpha } n(\alpha)^{2 s} \|  z_\alpha \|_0\Big)^{\frac12}\,.
\end{equation*}
By using the dyadic property \eqref{dya}, one has that the norm 
$\bral\cdot \brar_s$ is equivalent to the $\ell^2_s$-norm $\| \cdot \|_s$.

Consider now the normal form part of equation \eqref{eq:NF.2}, namely
\begin{equation}
  \label{NF.3}
\dot z^{\perp}=\Lambda z^{\perp}+\Pi^{\perp}X_{Z_2}(z^{\leq},
z^{\perp})\ ;
  \end{equation}
by \eqref{proiettore Pi alpha}, \eqref{spectral decomposition Bourgain-Delort}, 
\eqref{decomposizione cal Z (t)}, 
it is block diagonal, namely it is equivalent to the decoupled system 
\begin{equation*}%\label{eq finito dimensionali}
\partial_t z_\alpha(t) = - \ii \Lambda z_\alpha(t) + \ii {\cal
  Z}_\alpha(t) z_\alpha(t)\ . 
\end{equation*}
Since ${\cal Z}_\alpha$ is self-adjoint, one immediately has that 
\begin{equation}\label{stima u alpha L2}
\| z_\alpha(t) \|_{0} = \| z_\alpha(t_0) \|_{0}, \quad \forall t, t_0
\in [- T_R, T_R]\,, \quad \forall \alpha\,.
\end{equation}
Therefore, for any $t \in [- T_R, T_R]$, for the solution of
\eqref{NF.3} one has 
\[
\begin{aligned}
\| z(t) \|_s^2 & {\lesssim_s} \, \bral z(t)\brar_s^2 \, \lesssim_s \,
\sum_{\alpha } n(\alpha)^{2 s} \| z_\alpha(t) \|_{0}^2  
\\&  
\stackrel{\eqref{stima u alpha L2}}{\lesssim_s} 
\sum_{\alpha }
n(\alpha)^{2 s} \| z_\alpha(t_0) \|_{0}^2 \lesssim_s \, \bral z_0\brar_s^2\,
{\lesssim_s} 
\| z_0\|_s^2 \, ,
\end{aligned}
\]
so that, denoting by $\cU(t,\tau)$ the flow map of \eqref{NF.3}, one
has
\begin{equation*}
\left\|\cU(t,\tau)z_0\right\|_{s}\leq
K_2\left\|z_0\right\|_s\ ,\ \forall t\ .
  \end{equation*}
Consider now \eqref{eq:NF.2}. Using Duhamel formula one gets
\[
z^{\perp}(t)=\cU(t,0)z_0+\int_0^t\cU(t,\tau)\Pi^{\perp}X_{\cR}
(z^{\leq}(\tau),z^{\perp }(\tau))d\tau\ ,
\]
from which (using also \eqref{rtk.1fin}) one deduces
\[
\left\|z^{\perp}(t)\right\|_s\leq
K_2\left\|z^{\perp}_0\right\|_s+\int_0^tK_rR^{r+1}d\tau \ ,
\]
and the estimate \eqref{grandi.1}.
\qed

\begin{proof}[{\bf Conclusion of the proof of Theorem \ref{main.abs}}]
By Lemmas \ref{piccoli}, \ref{grandi}
(see estimates \eqref{bassi.1}, \eqref{grandi.1}) we have that the claim \eqref{claimR}
holds. By a standard bootstrap argument one can show that (recall \eqref{escapetime}, \eqref{marconi1})
$T_{R}\gtrsim \epsilon^{-r}$. This implies the thesis.
\end{proof}

\section{Applications}\label{appli}

Let $\be_1,...,\be_d$ be a basis of $\R^d$ and let
\begin{equation}
  \label{gamma}
\Gamma:=\Big\{x\in\R^d\ :\ x=\sum_{j=1}^d2\pi n_j\be_j\ ,\quad n_j\in\Z\Big\}
  \end{equation}
be a maximal dimensional lattice. We denote
$\T^d_\Gamma:=\R^d/\Gamma$.

To fit our scheme it is convenient to introduce in $\T^d_\Gamma$ the
basis given by $\be_1,...,\be_d$, so that the functions turn out to be
defined on the standard torus $\mathbb{T}^d:=\mathbb{R}^{d}/(2\pi\Z)^d$, but endowed by the metric
$\tg_{ij}:=\be_j\cdot\be_j$. In particular the Laplacian turns out to be
\begin{equation}\label{deltaGG}
\Delta_g:=\sum_{l,n=1}^{d}g_{ln}\partial_{x_{l}}\partial_{x_{n}}\ ,\qquad x=(x_1,\ldots,x_{d})\in \T^{d}\,,
\end{equation}
where $g_{ln}$ is the inverse of the matrix $\tg_{ij}$. The positive
definite symmetric quadratic form of equation \eqref{modulo} is then defined by
$$
g(k,k):=\sum_{l,n=1}^{d}g_{ln}k_lk_n\,,\qquad \forall\, k\in \Z^{d}\,.
$$
The coefficients $g_{ln}$, $l,n=1,\ldots, d$, of the metric $g$ above 
can be seen as parameters that will be chosen 
in the set we now introduce. We also assume the symmetry $g_{i j} = g_{j i}$ for any $i, j = 1, \ldots, d$, hence we identify the metric $g$ with $(g_{i j})_{i \leq j}$, namely we identify the space of symmetric metrics with $\R^{\frac{d(d + 1)}{2}}$. We denote by $\| g \|_{2}^2 := \sum_{i , j} |g_{i j}|^2$
\begin{definition}\label{parametrimetrica}
Consider the open set
\[
\begin{aligned}
\mathcal{G}_{0}&:=\left\{
\left(g_{ij}\right)_{i\leq j}\in \R^{\frac{d (d + 1)}{2}}\; : \; \inf_{x \neq 0} \frac{g(x, x)}{|x|^2} > 0
\right\}\,. 
%\mathcal{G}_{2}&:=
%\left\{
%\left(g_{ij}\right)_{i,j=1}^d\in \R^{d^2}\; :\;
%\inf_{j\in\Z^d\setminus\left\{0\right\}}\left\| j\right\|_g>\frac{1}{K}
%\right\}\,.
\end{aligned}
\]
Fix $\tau_{*}:= \frac{d(d + 1)}{2} +1$
%\red{$\tau_* :=  \frac{d(d - 1)}{2} + 1$ Sicuri non sia $\tau_* >  \frac{d(d + 1)}{2} -1$.} 
We then define the set of admissible metrics as follows.
$$
{\cal G} := \cup_{\gamma > 0} {\cal G}_\gamma
$$
where 
$$
\begin{aligned}
{\cal G}_\gamma & := \Big\{ g \in {\cal G}_0  : \Big| \sum_{i \leq j} g_{i j} \ell_{i j} \Big| \geq \dfrac{\gamma}{\Big(\sum_{i \leq j} |\ell_{i j}| \Big)^{\tau_*}}  \\
& \qquad \forall \ell \equiv (\ell_{i j})_{i \leq j} \in \R^{\frac{d(d + 1)}{2}} \setminus \{ 0 \}\Big\}
\end{aligned}
$$
\end{definition}

\begin{remark}\label{remark admissible set}
The set ${\cal G}_\gamma$ above satisfies a diophantine estimate $|({\cal G}_0 \cap B_R) \setminus {\cal G}_\gamma| \lesssim \gamma$ ($B_R$ is the ball in $\R^{\frac{d(d + 1)}{2}}$), implying that ${\cal G}$ has full measure in ${\cal G}_0 $ (we denote by $| \cdot |$ the Lebesgue measure). We also point out that in Section \ref{Schrodinger potenziale}, we only take the metric $g \in {\cal G}_0$ and we shall use the convolution potential in order to impose the non-resonance conditions. For the other applications, namely in sections \ref{section beam}, \ref{sec:QHD}, \ref{plane} we shall use that the metric $g$ is of the form $g = \beta \bar g$, with $\bar g$ in the set of the admissible metrics ${\cal G}$. We then use the parameter $\beta$, in order to verify the non-resonance conditions required. 
% is e in Definition \ref{parametrimetrica}
%can be viewed as a subset of $\R^{\frac{d(d+1)}{2}}$ taking as
%independent variables the $g_{ij}$'s with $i\leq j$. Accordingly,
%  we endow $\cG$ with the Lebesgue measure.
\end{remark}

\subsection{Schr\"odinger equations with convolutions potentials}\label{Schrodinger potenziale}
We consider Schr\"odinger equations of the form 
\begin{equation}\label{convoluzione}
\ii \partial_t \psi = - \Delta_g \psi + V * \psi + f(|\psi|^2)\psi , \quad x \in \T^d
\end{equation}
where $\Delta_{g}$ is in \eqref{deltaGG} with $g\in \mathcal{G}_0$ (see Def. \ref{parametrimetrica}),
$V$ is a potential, $*$ denotes the convolution
 and the nonlinearity $f$ is of class $C^\infty(\R,\R)$ in a
neighborhood of the origin and $f(0) = 0$.  
Equation \eqref{convoluzione}
is Hamiltonian with Hamiltonian function
\begin{equation}
  \label{ham.schro}
H=\int_{\T^d}\left(\nabla \psi \cdot \nabla \varphi+\varphi(V*\psi)+F(\psi \varphi)\right)dx
\end{equation}
where $F$ is a primitive of $f$ and $\varphi$ is a variable conjugated
to $\psi$. To get equation \eqref{convoluzione} one has to restrict to
the invariant manifold $\varphi=\overline\psi$. 

Fix $n\geq 0$ and $R>0$, then the potential $V$ is chosen in the space $\cV$
given by
\begin{equation}
  \label{cv}
\cV:=\left\{V(x)=\frac{1}{|\T^d|_g}\sum_{k\in\Z^d}\hat V_ke^{ik\cdot
  x}\ :\ %V'_k:=
  \hat V_k\langle
k\rangle^n\in\left[-\frac{1}{2},\frac{1}{2} \right]\ ,\ \forall
k\in\Z^d  \right\}\ ,
  \end{equation}
which we endow with the product probability measure. Here and below
$|\T^d|_g$ is the measure of the torus induced by the metric $g$.

\begin{theorem}
  \label{nsl.esti}
There exists a set $\cV^{(res)}\subset\cV$ with zero measure such that for any
$V\in \cV\setminus \cV^{(res)}$ the following holds.
For any $r\in\N$, there exists $ s_r>d/2$ such that for any
$s>s_r$ there is 
 $\epsilon_s>0$ and $C>0$ such that if the initial datum for
\eqref{convoluzione} belongs to $H^s$ and fulfills
$\epsilon:=\left|\psi\right|_s<\epsilon_s$ then
\begin{equation*}
 % \label{esti.schro}
\left\|\psi(t)\right\|_s\leq C\epsilon\,,\quad \text{for\ all}\,\quad 
\left|t\right|\leq C\epsilon^{-r}\,.
  \end{equation*}
  \end{theorem}

We are now going to prove this theorem.
To fit our scheme simply introduce
the Fourier coefficients 
\begin{equation*}%\label{fou.schro}
\psi(x)=\frac{1}{|\T^d|_g^{1/2}} \sum_{j\in\Z^d}u_{(j,+)}e^{ij\cdot x} \ ,
\quad  
\varphi(x)=\frac{1}{|\T^d|_g^{1/2}} \sum_{j\in\Z^d}u_{(j,-)}e^{-ij\cdot x}\ .
\end{equation*}
 In these variables the equation \eqref{convoluzione} takes the form
\eqref{Ham.eq.1} 
with $H=H_0+P$, $H_0$ of the form \eqref{H0} with
frequencies
\begin{equation}
  \label{fre.schro}
\omega_j:=\left|j\right|_g^2+\hat V_j\ ,
\end{equation}
and $P$
obtained by substituting in the $F$ dependent term of the Hamiltonian
\eqref{ham.schro}. It is easy to see that the perturbation is of class
$\cP$ of Definition \ref{funzP}.

In order to apply our abstract Birkhoff normal form theorem, we only
need to verify the Hypotheses \ref{hypo1}, \ref{hypo2}.
%(F.1)-(F.3) and (NR.1), (NR.2). 
The hypothesis (F.1) in Hyp. \ref{hypo1} holds trivially with $\beta=2$ using \eqref{fre.schro}.

The hypothesis (F.3) follows by the generalization of the
Bourgain's Lemma proved in \cite{BerMas}. Precisely we now prove the
following lemma.

\begin{lemma}
  \label{parti.nls}
The assumption (F.3) of Hypothesis \ref{hypo1} holds.
  \end{lemma}
\proof
Let
$\Omega_\alpha$ be the partition of $\Z^d$ constructed in Theorem $2.1$
 of \cite{BerMas}. It 
satisfies the properties 
\begin{equation*}%\label{clusters Berti Maspero}
\begin{aligned}
 ||j|^2 - |j'|^2|+ {|j-j'|} &\geq C_0 (|j|^\delta + |j'|^\delta)\,, \qquad j \in \Omega_\alpha\,, 
 \quad j' \in \Omega_\beta\,, \quad \alpha \neq \beta\,, 
 \\
 \max_{j \in \Omega_\alpha} |j|& \lesssim \min_{j \in \Omega_\alpha} |j|\,, 
\end{aligned}
\end{equation*}
for some $C_0 > 0$ and $\delta \in (0, 1)$.
Clearly, one has that if $j \in \Omega_\alpha, j' \in \Omega_\beta$ 
with $\alpha \neq \beta$, one has that 
\[
\begin{aligned}
|\omega_j - \omega_{j'}|+ {|j-j'|}  & = ||j|^2 - |j'|^2 + \widehat V_j - \widehat V_{j'}|+ {|j-j'|}  
\\
&\geq ||j|^2 - |j'|^2| +{|j-j'|}- 2 \sup_{j \in \Z^d} |\widehat V_j|  
 \\
 & \geq
||j|^2 - |j'|^2| + {|j-j'|}- 1 
\\ & 
\geq C_0  (|j|^\delta + |j'|^\delta) - 1  
\geq C_0  (|j|^\delta + |j'|^\delta)/2\,,
\end{aligned}
\]
provided $|j|^\delta + |j'|^\delta \geq \frac{2}{C_0}$, which is
verified when $|j| + |j'| \geq C(\delta, C_0)$ for some constant
$C(\delta, C_0) > 0$. 
\qed

It remains to verify conditions (F.2) in Hyp. \ref{hypo1} and 
(NR.1), (NR.2) in Hyp. \ref{hypo2}.
%We now proceed by verifying the hypotheses (F.2) and (NR).

\smallskip
\noindent
First we define $[j]:=j$.
Given $r$
and $N$ we define
\begin{align*}
%  \label{insiemi}
  \Z^d_N:=\left\{j\in\Z^d\ :\ |j|\leq N\right\} \ ,
  \\
 % \label{insiemi.1}
  \cK^r_N:=\left\{k\in\Z^{\Z^d_N}\ :\ 0\not=|k|\leq r\right\}\ ,
\end{align*}
and remark that its cardinality $\#\cK^r_N\leq N^{dr}$. For
$k\in\Z^{\Z^d_N}$, consider
$$
\cV^N_k(\gamma):=\left\{V\in\cV\ :\ \left|\omega\cdot
k\right|<\gamma\right\}\ .
$$
\begin{lemma}
  \label{meas.nls}
  One has 
  \begin{equation}
    \label{meas.nls.1}
\left| \cV^N_k(\gamma) \right|\leq 2\gamma N^n
  \end{equation}
  with $n$ the number in the definition of $\cV$ in \eqref{cv}.  
  \end{lemma}
\proof If $\cV^N_k(\gamma) $ is empty there is nothing to
prove. Assume that $\tilde V\in \cV^N_k(\gamma)$. 
Since $k\not=0$, there exists 
$ \bar j$ such that $k_{\bar j}\not=0$ and thus 
$\left|k_{\bar j}\right|\geq 1$; so we have
\[
\left|\frac{\partial \omega\cdot k}{\partial \hat V_{\bar j}}\right|\geq 1\ .
\]
It means that if $\cV^N_k(\gamma) $ is not empty it is contained in
the layer
\[
\left|\widehat{\tilde V_{\bar j}}\null'-\hat V'_{\bar j}\right|\leq \gamma\ ,
\]
whose measure is
$\gamma\langle \bar j\rangle^n\leq 2\gamma N^n$.
This implies \eqref{meas.nls.1}.
\qed
\begin{lemma}
  \label{tutto.res}
For any $r$ there exists $\tau$ and a set $\cV^{(res)}\subset\cV$ of zero
measure, s.t., if $V\in\cV\setminus \cV^{(res)}$ there exists
$\gamma>0$ s.t. for all $N\geq 1$ one has
\begin{equation*}
%  \label{nonres.beam}
\left|\omega\cdot k\right|\geq\frac{\gamma}{N^\tau}\ ,\ \forall k\in
\cK^r_N\ .
  \end{equation*}
  \end{lemma}
\proof From Lemma \ref{meas.nls} it follows that the measure of the set
\[
\cV^{(res)}(\gamma):=\bigcup_{N\geq 1}
\bigcup_{k\in\cK^r_N}
\cV^N_k\left(\frac{\gamma}{ N^{dr+2}}\right)
\]
is estimated by a constant times $\gamma$. It follows that the set
\[
\cV^{(res)}:=\cap_{\gamma>0}\cV^{(res)}(\gamma)
\]
has zero measure and with this definition the lemma is proved. 
\qed

\subsection{Beam equation}\label{section beam}

In this section we study the beam equation 
\begin{equation}
  \label{beam}
\psi_{tt}+\Delta^2_g\psi+m\psi=-\frac{\partial F}{\partial
  \psi}+\sum_{l=1}^d \partial_{x_l}\frac{\partial
  F}{\partial (\partial_l \psi)}\ ,
\end{equation}
with $F(\psi,\partial_{x_1}\psi,...,\partial_{x_d}\psi)$ a function of class
$C^{\infty}(\R^{d+1};\R)$ in a neighborhood of the origin and having
a zero of order 2 at the origin. 

Introducing the variable $\varphi=\dot{\psi}\equiv\psi_{t}$,
it is well known that \eqref{beam} can be seen as an Hamiltonian system in the 
variables $(\psi,\varphi)$ 
% this is a Hamiltonian equation 
 with Hamiltonian function
\begin{equation}
  \label{ham.beam}
H(\psi,\varphi):=\int_{\T^d}\left(\frac{\varphi^2}{2}
+\frac{\psi(\Delta_{g}^2+m)\psi}{2} 
+F(\psi,\partial_1\psi,...,\partial_d\psi)
\right)dx \ . 
\end{equation} 
In order to fulfill the diophantine non-resonance conditions 
on the
frequencies we need to make some restrictions on the metric $g$ whereas, we only require that the mass $m > 0$ is strictly positive. 
 More precisely, we consider $\bar g$ be a metric in the set of the admissible metrics ${\cal G}$ given in the definition \ref{parametrimetrica}. We consider a metric $g$ of the form 
\begin{equation}\label{bla bla metrica constrained}
g = \beta \bar g, \quad \beta \in {\cal B} := (\beta_1, \beta_2), \quad 0 < \beta_1 < \beta_2 < + \infty\,.
\end{equation}
we shall use the parameter $\beta$ in order to tune the resonances and to impose the non-resonance conditions required in order to apply Theorem \ref{main.abs}. The precise statement of the main theorem of this section is the following one. 
\begin{theorem}
  \label{main.beam}
Let $\overline g \in {\cal G}$, There exists a set of zero measure ${\cal B}^{(res)}\subset {\cal B}$ such that if
$\beta\in {\cal B} \setminus {\cal B}^{(res)}$ then for all $ r\in\N$ there
exist $s_r>d/2$ such that the following holds. 
For any $s>s_r$ there exist $ \epsilon_{rs},c,C$ 
such that
%with the
%following property: 
if the initial datum for \eqref{beam} fulfills
\begin{equation}
  \label{beam.dat}
\epsilon:=\left\|(\psi_0,\dot\psi_0)\right\|_s:=
\left\|\psi_0\right\|_{H^{s+2}}+
\left\|\dot \psi_0\right\|_{H^{s}}<\epsilon_{sr}\ , 
\end{equation}
then the corresponding solution satisfies 
\begin{equation}
  \label{stima.sol}
\left\| (\psi(t),\dot\psi(t))\right\|_{s}\leq
C\epsilon\,,\quad  \text{for}\quad \left|t\right|\leq c\epsilon^{-r}\,.
%\frac{c}{\epsilon^{r}}
\end{equation}
\end{theorem}
We actually state also a corollary which state that there exists a full measure set of metrics (not only constrained to a given direction $\overline g$) for which the statements of Theorem \ref{main.beam} hold. 
Let $0 < \beta_1 < \beta_2$ and define 
\begin{equation}\label{cal G beta 1 beta 2}
{\cal G}_0(\beta_1, \beta_2) := \Big\{ g \in {\cal G}_0 : \beta_1 \leq \| g \|_2 \leq \beta_1 \Big\}\,. 
\end{equation}
where ${\cal G}_0$ is given in the definition \ref{parametrimetrica}. 
\begin{corollary}\label{remark misura piena beam}
There exists a zero measure set ${\cal G}^{(res)}_{\beta_1, \beta_2} \subseteq {\cal G}_0(\beta_1, \beta_2)$ such that for any $g \in {\cal G}_0(\beta_1, \beta_2) \setminus {\cal G}^{(res)}_{\beta_1, \beta_2}$ the conclusion of theorem \ref{main.beam} hold. 
\end{corollary}
\noindent{\it Proof of Corollary \ref{remark misura piena beam}}.
To shorten notations in this proof, we denote by $n := \frac{d(d + 1)}{2}$. For any $\beta_1 \leq \beta \leq \beta_2$, we denote by 
$\sigma_\beta$ the surface $n - 1$ dimensional measure on the sphere $ \partial B_\beta := \{\| g \|_2 = \beta \}$. We now prove the following two claims
\begin{itemize}
\item{\bf Claim 1.} One has that the surface measure of all diophantine metrics ${\cal G}$ in ${\cal G}_0$ with norm equal 1 has full surface measure in ${\cal G}_0 \cap \partial B_1$, namely $\sigma_1({\cal G} \cap \partial B_1) = \sigma_1({\cal G}_0 \cap \partial B_1)$. 
\item{\bf Claim 2.} Let $\overline g \in {\cal G} \cap \partial B_1$ and let ${\cal B}_{\overline g} \subset (\beta_1, \beta_2)$ the full measure set provided in Theorem \ref{main.beam}. We shall prove that 
$$
{\cal G}_{\beta_1, \beta_2}^{(nr)} := \bigcup_{\overline g \in \partial B_1 \cap {\cal G}} {\cal B}_{\overline g}
$$
has full measure in ${\cal G}_0(\beta_1, \beta_2)$. 

\noindent
{\sc Proof of Claim 1.} Let $E \subset \partial B_1$. Then the set $$\beta E := \big\{ \beta x : x \in E \big\} \subset \partial B_\beta$$ and, by standard scaling properties, 
\begin{equation}\label{invariance bla}
\sigma_\beta(\beta E) = C_n \beta^{n - 1} \sigma_1(E)\quad \text{for some constant} \quad C_n > 0\,. 
\end{equation}
By \eqref{cal G beta 1 beta 2} and Remark \eqref{remark admissible set}, the set ${\cal G}_{\beta_1, \beta_2} := {\cal G}_0(\beta_1, \beta_2) \cap {\cal G}$ has full measure in the open set ${\cal G}_0(\beta_1, \beta_2)$. By Fubini one has 
\begin{equation}\label{fuffa scaling 0}
\begin{aligned}
|{\cal G}_0(\beta_1, \beta_2)| & = \int_{\beta_1}^{\beta_2} \sigma_{\beta}({\cal G}_0 \cap \partial B_\beta)\, d \beta \\
&  \stackrel{\eqref{invariance bla}}{=}  C_n \int_{\beta_1}^{\beta_2} \beta^{n - 1} \sigma_1({\cal G}_0 \cap \partial B_1)\, d \beta \\
& = \frac{C_n(\beta_2^n - \beta_1^n)}{n} \sigma_1\Big( {\cal G}_0 \cap \partial B_1 \Big)
\end{aligned}
\end{equation}
and similarly
\begin{equation}\label{fuffa scaling 1}
\begin{aligned}
|{\cal G}_{\beta_1, \beta_2}| & = \int_{\beta_1}^{\beta_2} \sigma_{\beta}({\cal G} \cap \partial B_\beta)\, d \beta \\
&  \stackrel{\eqref{invariance bla}}{=}  C_n \int_{\beta_1}^{\beta_2} \beta^{n - 1} \sigma_1({\cal G} \cap \partial B_1)\, d \beta \\
& = \frac{C_n(\beta_2^n - \beta_1^n)}{n} \sigma_1\Big( {\cal G}\cap \partial B_1 \Big)\,. 
\end{aligned}
\end{equation}
Since $|{\cal G}_0(\beta_1, \beta_2)| = |{\cal G}_{\beta_1, \beta_2}|$, by comparing \eqref{fuffa scaling 0}, \eqref{fuffa scaling 1}, one immediately gets that $\sigma_1({\cal G} \cap \partial B_1) = \sigma_1({\cal G}_0 \cap \partial B_1)$. 

\noindent
{\sc Proof of claim 2.} By Fubini, the Lebesgue measure $|{\cal G}_{\beta_1, \beta_2}^{(nr)}|$ is 
$$
\begin{aligned}
|{\cal G}_{\beta_1, \beta_2}^{(nr)}| &= \int_{{\cal G} \cap \partial B_1} |{\cal B}_{\overline g}|\, d \sigma_1(\overline g) = (\beta_2 - \beta_1) \int_{{\cal G} \cap \partial B_1}\, d \sigma_1(\overline g) = (\beta_1 - \beta_2) \sigma_1({\cal G} \cap \partial B_1) \\
& \stackrel{Claim\, 1}{=} (\beta_1 - \beta_2) \sigma_1(\partial B_1 \cap {\cal G}_0) = |{\cal G}_0(\beta_1, \beta_2)|\,.
\end{aligned}
$$
The claimed statement has then been proved. 
\end{itemize}
\qed

To prove Theorem \ref{main.beam} we first show how to fit our scheme and then
we prove that the Hypotheses of Theorem \ref{main.abs} are verified. 

To fit our scheme we first introduce new variables
\begin{align}
  \label{u+}
u_+(x):=\frac{1}{\sqrt 2}\left(\left(\Delta_{g}^2+m\right)^{1/4}\varphi+i
\left(\Delta_{g}^2+m\right)^{-1/4}\psi\right)\ ,\quad  
\\
u_-(x):=\frac{1}{\sqrt 2}\left(\left(\Delta_{g}^2+m\right)^{1/4}\varphi
-i \left(\Delta_{g}^2+m\right)^{-1/4}\psi\right)\ ,
\label{u-}
\end{align}
and consider their Fourier series, namely, for $\sigma=\pm1$
\begin{equation*}
 % \label{fou.beam}
u_\sigma(x)=\frac{1}{|\T^d|_g^{1/2}}
\sum_{j\in\Z^d}u_{(j,\sigma)}e^{\sigma ij\cdot
x}   \ .
\end{equation*}
In these variables the beam equation \eqref{beam} 
takes exactly the form
\eqref{Ham.eq.1} with $H=H_0+P$, $H_0$ of the form \eqref{H0} with
frequencies
\begin{equation}
  \label{fre.beam}
\omega_j:=\sqrt{\left|j\right|^4_g+m}\ .
\end{equation}
and $P$
obtained by substituting \eqref{u+}-\eqref{u-} in the $F$ dependent term of the Hamiltonian
\eqref{ham.beam}.
Thanks to the regularity assumption on $F$,
it is easy to see that the perturbation $P$ is of class $\cP$. 

The verification of (F.3) in Hyp. \ref{hypo1} 
goes exactly as in the case of the
Schr\"odinger equation, since the asymptotic of 
$\omega_j $ in \eqref{fre.beam} 
is $\omega_j = |j|^2_g + O(1)$. The asymptotic condition (F.1) is also
trivially fulfilled with $\beta=2$.
The main point is to verify the non-resonance conditions (F.2) and
the conditions 
(NR.1), (NR.2) in Hyp. \ref{hypo2}. 
This will occupy the rest of this subsection.

First of all we remark that the equivalence classes of Definition
\ref{equi} are simply defined by  
\begin{equation*}
 % \label{classi.beam}
[j]\equiv\left\{i\in\Z^d\ :\ |i|_g=|j|_g\right\}\ .
\end{equation*}
Now, recall that $g = \beta \overline g$ with $\overline g \in {\cal G}$ and $\beta \in {\cal B} = [\beta_1, \beta_2]$. One can easily verify that 
\begin{equation}\label{j g bar g}
|j|_g = \beta |j|_{\overline g}
\end{equation}
implying that $|j|_g = |k|_g$ if and only if $|j|_{\overline g} = |k|_{\overline g}$. Hence the equivalence class $[j]$ is 
$$
[j]\equiv\left\{i\in\Z^d\ :\ |i|_{\bar g}=|j|_{\bar g}\right\}\ .
$$

We are going to prove the following Lemma
\begin{lemma}
  \label{non.beam}
Let $\overline g \in {\cal G}$. There exists a set ${\cal B}^{(res)}\subset {\cal B}$ of zero measure, s.t., if
$\beta \in {\cal B} \setminus {\cal B}^{(res)}$ then (NR.1), (NR.2) and (F.2)
hold for the metric $g = \beta \, \overline g$. 
\end{lemma}

We first need a lower bound on the distance between points with
different modulus in $\Z^d$. The following lemma holds

\begin{lemma}
  \label{dista}
Fix any $N>1$ and let $\overline g \in {\cal G}$, $\beta \in {\cal B} = (\beta_1, \beta_2)$, $g = \beta \overline g$. One has that if $j, k \in \Z^d$ such that $|\ell|, |k| \leq N$, $|j|_g \neq |k|_g$, then there exists $\gamma > 0$ and a constant $C (\beta_1)$ such that 
$$
||k|_g^2 - |\ell|_g^2| \geq \frac{C(\beta_1)\gamma}{N^{2\tau_*}} 
$$
where $\tau_*$ is given in the definition \ref{parametrimetrica}. 
\end{lemma}
\proof 
By recalling the Definition \ref{parametrimetrica} of the admissible set ${\cal G}$, since $\overline g \in {\cal G}$, one has that there exists $\gamma > 0$ such that 
$$
\Big| \sum_{i \leq j} \overline g_{i j} \ell_{i j} \Big| \geq \dfrac{\gamma}{\Big( \sum_{i \leq j} |\ell_{i j}|\Big)^\tau} \quad \forall \ell = (\ell_{i j})_{i \leq j} \in \R^{\frac{d (d + 1)}{2}} \setminus \{ 0 \}\,. 
$$
By \eqref{j g bar g}, since $g = \beta \overline g$, one has that 
\begin{equation}\label{platini0}
\begin{aligned}
 |\left|k\right|_g^2-\left|  \ell \right|_g^2 | & = \beta  \Big| \sum_{ij}\overline g_{ij}\left(k_i k_j-  \ell_i \ell_j\right) \Big| \\
& \geq \beta_1 \dfrac{\gamma}{\Big( \sum_{i , j} |k_i k_j - \ell_i \ell_j| \Big)^\tau}\,.
\end{aligned}
\end{equation}
Since $|\ell|, |k| \leq N$, one has the following chain of inequalities: 
$$
\begin{aligned}
 \sum_{i , j} |k_i k_j - \ell_i \ell_j|  & \leq \sum_{i j} (|\ell_i| |\ell_j| + |k_i| |k_j| ) \lesssim |\ell|^2 + |k|^2 \stackrel{|\ell|, |k| \leq N}{\lesssim} N^2\,.
\end{aligned}
$$
The latter inequality, together with \eqref{platini0} imply that there exists a constant $C(\beta_1)$ such that 
$$
 |\left|k\right|_g^2-\left|  \ell \right|_g^2| \geq \dfrac{C(\beta_1) \gamma}{N^{2 \tau}}
$$
uniformly on $\beta \in (\beta_1, \beta_2)$. 
The claimed statement has then been proved. \qed

Using the property \eqref{j g bar g}, one can easily verify that  the frequencies $\omega_j \equiv \omega_j(\beta)$ assume the form 
\begin{equation}\label{forma frequenze per non risonanza}
\omega_j(\beta) = \beta^2 \Omega_j, \quad \Omega_j :=  \sqrt{|j|^4_{\overline g} + \frac{m}{\beta^4}}\,.
\end{equation}
Since $\beta_2 \geq \beta \geq \beta_1 > 0$, 
$$
\Big|\sum_{i = 1}^r \sigma_i \omega_{j_i} \Big| =  \beta^2 \Big| \sum_{i = 1}^r \sigma_i \Omega_{j_i}\Big| \geq \beta_1^2  \Big| \sum_{i = 1}^r \sigma_i \Omega_{j_i}\Big|
$$
one can verify non resonance conditions on $\Omega_j$. Since the map 
\begin{equation}\label{mappa beta gamma beam}
(\beta_1, \beta_2) \to (\zeta_1, \zeta_2) := \Big(\frac{m}{\beta_2^4}, \frac{m}{\beta_1^4} \Big)\,, \qquad  \beta \mapsto \zeta := \frac{m}{\beta^4}
\end{equation}
is an analytic diffeomorphism, we can introduce $\zeta = m/ \beta^4$ as parameter in order to tune the resonances. Hence we verify non resonance conditions on the frequencies 
\begin{equation}\label{Omega j beam}
\Omega_j (\zeta) = \sqrt{|j|_{\bar g}^4 + \zeta}, \quad j \in \Z^d\,. 
\end{equation}

\begin{lemma}
  \label{determinante}%[(Essentially Lemma 9 of \cite{Bam09})]
Let $\bar g \in {\cal G }$. For any $K\leq N$, consider $K$ indexes $\bji 1,...,\bji K$ with
$\left|\bji 1\right|_g<...\left|\bji K\right|_g\leq N$; and 
consider the determinant
\begin{equation*}
D:=\left|
\begin{matrix}
\Omega_{\bji 1}& \Omega_{\bji 2}&.& .&.&\Omega_{\bji K} 
\\
\partial_\zeta \Omega_{\bji 1} &\partial_\zeta \Omega_{\bji 2} & .& .&.&\partial_\zeta \Omega_{\bji K}
\\
.& .& .& .& .&.
\\
.& .& .& .& .&.
\\
\partial_\zeta^{K - 1} \Omega_{\bji 1}& \partial_\zeta^{K - 1} \Omega_{\bji 2}& .& .&.&\partial_\zeta^{K - 1} \Omega_{\bji K}
\end{matrix}
\right|
\end{equation*}
There exists $C>0$ s.t.
\begin{equation*}
D
\geq
\frac{C}{N^{\eta K^2}}\,,
\end{equation*}
for some constant $\eta \equiv \eta_d > 0$ depending only on the dimension $d$.
%There exists $\eta>0$ such that
%\begin{eqnarray}
%D= C\left(\prod_{l}\omega_{i_l}^{-2K+1} \right) 
%  \left(\prod_{1\leq l<k\leq
%  K} (|j_l|^2_g - |j_k|^2_g )\right)
%\label{v.401}
%\geq
%\frac{C}{N^{\eta K^2}}\,.
%\end{eqnarray}
%{\color{red} NON SONO SICURO DELLA COSTANTE. SI PUO METTERE $\geq N^{\eta K^2}$ for some constant $\eta \equiv \eta_d$ depending only on the dimension. NELLA APPLICAZIONE IDRODINAMICA MI SONO FATTO LA DIM E VENIVA UNA COSTANTE DIVERSA.}
\end{lemma}
The proof was given in \cite{Bam09}. For sake of completeness we insert it. 
\begin{proof}
For any $i = 1, \ldots, K$, for any $n = 0, \ldots, K - 1$, one computes 
$$
\partial_\zeta^n \Omega_{j_i}(\zeta) = C_n (|j_i|_{\bar g}^4 + \zeta)^{\frac12 - n}
$$
for some constant $C_n \neq 0$. This implies that 
$$
D \geq C \prod_{i = 1}^K \sqrt{|j_i|_{\bar g}^4 + \zeta}  |{\rm det}(A)|
$$
where the matrix $A$ is defined as 
$$
A = \begin{pmatrix}
1& 1 &.& .&.&1
\\
x_1&x_2&.& .&.&x_K 
\\
.& .& .& .& .&.
\\
.& .& .& .& .&.
\\
x_1^{K  - 1}& x_2^{K - 1}&.& .&.&x_K^{K - 1} 
\end{pmatrix}
$$
where 
$$
x_i := \frac{1}{|j_i|_{\bar g}^4 + \zeta}, \quad i = 1, \ldots, K\,. 
$$
This is a Van der Monde determinant. Thus we have 
$$
\begin{aligned}
|{\rm det}(A)|  & = \prod_{1 \leq i < \ell \leq K} |x_i - x_\ell| = \prod_{1 \leq i < \ell \leq K} \Big| \frac{1}{|j_i|_{\bar g}^4 + \zeta} - \frac{1}{|j_\ell|_{\bar g}^4 + \zeta} \Big| \\
& \geq \prod_{1 \leq i < \ell \leq K}  \frac{||j_i|_{\bar g}^4 - |j_\ell|_{\bar g}^4|}{||j_i|_{\bar g}^4 + \zeta| ||j_\ell|_{\bar g}^4 + \zeta|}  \\
& \geq \prod_{1 \leq i < \ell \leq K}  \frac{(|j_i|_{\bar g}^2 + |j_\ell|_{\bar g}^2) ||j_i|_{\bar g}^2 - |j_\ell|_{\bar g}^2|}{||j_i|_{\bar g}^4 + \zeta| ||j_\ell|_{\bar g}^4 + \zeta|}  \\
%\\& 
& \stackrel{Lemma \,\,\ref{dista}}{\geq}  C N^{- K^2(\tau_* + 4)}
\end{aligned}
$$
which implies the thesis.

\end{proof} 
Exploiting this Lemma, and following step by step the proof of Lemma
12 of \cite{Bam09} one gets 

\begin{lemma}
\label{v.001}
Let $\overline g \in {\cal G}$. Then and for any $r$ there
exists $\tau \equiv \tau_r$ with the following property: for any positive $\gamma$
small enough there exists a set $I_\gamma\subset (\zeta_1, \zeta_2)$ such
that $\forall \zeta \in I_\gamma$ one has that for any $N\geq 1$ and any
multi-index  $J_1,...,J_r$ with $|J_l|\leq N$ $\forall l$, one has
\begin{equation*}
%\label{v.002}
\sum_{l=1}^r\sigma_l\Omega_{j_l}\not =0\quad \Longrightarrow
\quad
\left|\sum_{l=1}^r\sigma_l\Omega_{j_l}\right|
\geq\frac{\gamma}{N^{\tau}}\ .
\end{equation*}
Moreover,
\begin{equation*}
%\label{v.003}
\left| (\zeta_1, \zeta_2) \setminus I_\gamma\right|\leq C \gamma^{1/r}\ .
\end{equation*} 
\end{lemma} 
%\begin{corollary}\label{corollario Omega j omega j}
%Let $\overline g \in {\cal G}$. Then and for any $r$ there
%exists $\tau \equiv \tau_r$ with the following property: for any positive $\gamma$
%small enough there exists a set ${\cal I}_\gamma\subset {\cal B}$ such
%that $\forall \beta \in {\cal I}_\gamma$ one has that for any $N\geq 1$ and any
%multiindex  $J_1,...,J_r$ with $|J_l|\leq N$ $\forall l$, one has
%\begin{equation*}
%%\label{v.002}
%\sum_{l=1}^r\sigma_l\omega_{j_l}\not =0\quad \Longrightarrow
%\quad
%\left|\sum_{l=1}^r\sigma_l\omega_{j_l}\right|
%\geq\frac{\gamma}{N^{\tau}}\ .
%\end{equation*}
%Moreover,
%\begin{equation*}
%%\label{v.003}
%\left| {\cal B}\setminus {\cal I}_\gamma\right|\leq C \gamma^{1/r}\ .
%\end{equation*} 
%
%\end{corollary}

\noindent
{\it End of the proof of Lemma \ref{non.beam}} Let $\gamma > 0$. By recalling the diffeomorphism \eqref{mappa beta gamma beam}, one has that the set 
$$
{\cal I}_\gamma := \{ \beta \in [\beta_1,  \beta_2] : m/\beta^4 \in I_\gamma \}\,.
$$
satisfies the estimate 
$$
|(\beta_1, \beta_2) \setminus {\cal I}_\gamma| \lesssim \gamma^{\frac1r}
$$
Now, if we take $\beta \in {\cal I}_\gamma$ and if $\sum_{i=1}^r\sigma_i\omega_{j_i}\not =0$, one has that 
$$
\begin{aligned}
\Big|\sum_{i = 1}^r \sigma_i \omega_{j_i} \Big| & =  \beta^2 \Big| \sum_{i = 1}^r \sigma_i \Omega_{j_i}\Big| \stackrel{\beta_1 \leq \beta \leq \beta_2}\geq \beta_1^2  \Big| \sum_{i = 1}^r \sigma_i \Omega_{j_i}\Big|  \\
& \geq \frac{\beta_1^2 \gamma}{N^\tau}\,. 
\end{aligned}
$$
By the above result,
one has  that, if
$$
\beta \in\bigcup_{\gamma>0}{\cal I}_\gamma\ ,
$$
then (NR.2) holds and furthermore $\bigcup_{\gamma>0}{\cal I}_\gamma$ has
full measure. Hence the claimed statement follows by defining ${\cal B}^{(res)} := {\cal B} \setminus \Big( \bigcup_{\gamma>0}{\cal I}_\gamma \Big)$. 
%Here we put in evidence the fact that the set depends on
%$g$. To conclude the proof of the theorem, just define
%$\cS^{(good)}:=\bigcup_{g\in\cG\setminus\cG^{(res)} }(g,I^g) $ and
%$\cS^{(res)}:=\left(\cG\times\cM\right)\setminus \cS^{(good)}$.
\qed

\subsection{The Quantum hydrodinamical system}\label{sec:QHD}

We consider the following quantum hydrodynamic system on an irrational torus $\T^d_\Gamma$
\begin{equation}\tag{QHD}\label{EK3}
\left\{\begin{aligned}
& \pa_{t}\rho=-\mathtt{m}\Delta_g\phi-{\rm div}(\rho\nabla_g\phi)\\
&\pa_{t}\phi=-\tfrac{1}{2}|\nabla_g\phi|^{2}- p(\mathtt{m}+\rho)
+\frac{\kappa}{\mathtt{m}+\rho}\Delta_g\rho
-\frac{\kappa}{2(\mathtt{m}+\rho)^{2}}
%\tfrac{1}{2}K'(\mathtt{m}+\rho)
|\nabla_g\rho|^{2}\,,
\end{aligned}\right.
\end{equation}
where $\mathtt{m}>0$, $\kappa>0$,   the function $p$ belongs to 
$C^{\infty}(\mathbb{R}_{+};\mathbb{R})$ and $p(\mathtt{m})=0$.
The function $\rho(t,x)$ is such that $\rho(t,x)+\mathtt{m}>0$ 
and it has zero average in $x$. The variable $x$ is on the irrational torus $\T^d$ (as in the previous two applications). We assume the conditions
\begin{equation}\label{elliptic}
p'(\mathtt{m})>0\,.  
\end{equation}
We shall use Theorem \ref{main.abs} in order to prove the following almost global existence result. 
In order to give a precise statement of the main result, 
we shall introduce the following notation. 
Given a function $u : \T^d \to \C$, we define
\begin{equation*}%\label{Pi 0 Pi 0 bot}
\Pi_0 u := \frac{1}{|\T^d_\Gamma|^{\frac12}} \int_{\T^d} u(x)\, d\,x\,, 
\quad \Pi_0^\bot := {\rm Id} - \Pi_0\,. 
\end{equation*}
Let $\bar g$ be a metric in the set of the admissible metrics ${\cal G}$ given in the definition \ref{parametrimetrica}. Exactly as in the case of the Beam equation, we consider a metric $g$ of the form 
\begin{equation}\label{bla bla metrica constrained QHD}
g = \beta \bar g, \quad \beta \in {\cal B} := (\beta_1, \beta_2), \quad 0 < \beta_1 < \beta_2 < + \infty\,.
\end{equation}
we shall use the parameter $\beta$ in order to tune the resonances and to impose the non-resonance conditions required in order to apply Theorem \ref{main.abs}. The precise statement of the long time existence for the QHD system is the following. 

\begin{theorem}\label{main-idrodinamica} 
Let $\bar g \in {\cal G}$. There exists a set of zero measure ${\cal B}^{(res)}\subset {\cal B}$, s.t. if
$\beta \in {\cal B} \setminus {\cal B}^{(res)}$ and $g = \beta \bar g$, then, $\forall r\geq 2$ there
exist $s_r$ and $\forall s>s_r$ $\exists \epsilon_{rs},c,C$ with the
following property. For any initial datum $(\rho_0, \phi_0) \in H^s(\T^d_\Gamma) \times H^s(\T^d_\Gamma)$ satisfying 
\[
\| \rho_0 \|_s + \| \Pi_0^\bot\phi_0 \|_s \leq \epsilon
\]
there exists a unique solution $t \mapsto (\rho(t), \phi(t))$ of the system \eqref{EK3} satisfying the bound 
\[
\| \rho(t) \|_s + \| \Pi_0^\bot \phi(t) \|_s \leq C \epsilon\,, 
\quad \forall\, |t| \leq c\epsilon^{-r}\,. %\frac{c}{\epsilon^r}\,.  
\]
%There exists $s_0\equiv s_0(d)\in\R$ such that for almost all 
%$\nu\in[1,{2}]^{d}$, for any $s\geq s_0$, $\mathtt{m}>0$, $\kappa>0$ there exist $C>0$,
% $\epsilon_0>0$ such that for any $0<\epsilon\leq \epsilon_0$
%we have the following. 
%For any initial data 
%$(\rho_0,\phi_0)\in H_0^{s}(\mathbb{T}_{\nu}^{d})\times H^{s}(\mathbb{T}_{\nu}^{d})$
%such that
%\begin{equation}\label{initialstima}
%%\frac{1}{\mathtt{m}}
%\|\rho_0\|_{H^{s}(\mathbb{T}_{\nu}^{d})}+%\frac{1}{\sqrt{\kappa}}
%\|\Pi_0^{\bot}\phi_0\|_{H^{s}(\mathbb{T}_{\nu}^{d})}\leq \epsilon\,,
%\end{equation}
%there exists a unique solution 
%of \eqref{EK3} with $(\rho(0),\phi(0))=(\rho_0,\phi_0)$
% such that 
%\begin{equation}\label{tesiKG}
%\begin{aligned}
%&(\rho(t),\phi(t))\in C^0\big([0,T_\epsilon);H^{s}(\T_\nu^d)\times H^{s}(\T_\nu^d)\big)
%\bigcap 
%C^1\big([0,T_\epsilon);H^{s-2}(\T_\nu^d)\times H^{s-2}(\T_\nu^d)\big)\,, \\
%& \sup_{t\in[0,T_\epsilon)}\Big(\|\rho(t,\cdot)\|_{H^{s}(\mathbb{T}_{\nu}^{d})} +
%\| \Pi_0^{\bot}\phi(t,\cdot)\|_{H^{s}(\mathbb{T}_{\nu}^{d})}    \Big)\leq C\epsilon\,, 
%\qquad T_\epsilon\geq  \epsilon^{- r}
%\end{aligned}
%\end{equation}
\end{theorem}
Arguing as in the proof of Corollary \ref{remark misura piena beam}, one can show 
\begin{corollary}\label{remark misura piena QHD}
Let $0 < \beta_1 < \beta_2$. There exists a zero measure set ${\cal G}^{(res)}_{\beta_1, \beta_2} \subseteq {\cal G}_0(\beta_1, \beta_2)$ (where ${\cal G}_0(\beta_1, \beta_2)$ is defined in \eqref{cal G beta 1 beta 2}) such that for any $g \in {\cal G}_0(\beta_1, \beta_2) \setminus {\cal G}^{(res)}_{\beta_1, \beta_2}$ the statements of theorem \ref{main-idrodinamica} hold. 
\end{corollary}
The key tool in order to prove the latter almost global existence result \ref{main-idrodinamica} is to use a change of coordinates (the so called Madelung transformation) which allows to reduce the system \eqref{EK3} to a semilinear Schr\"odinger type equation. We shall implement this in the next sections. 
%\begin{theorem}
%  \label{main.idrodinamica}
%There eists a set of zero measure $\cS^{(res)}\subset(\cG\times\cM)$, s.t. if
%$(g,m)\in (\cG\times\cM)\setminus\cS^{(res)} $ then, $\forall r\geq 3$ there
%exist $s_r$ and $\forall s>s_r$ $\exists \epsilon_{rs},c,C$ with the
%following property: if the initial datum for \eqref{beam} fulfills
%\begin{equation}
%  \label{beam.dat}
%\epsilon:=\left\|(\psi_0,\dot\psi_0)\right\|:=
%\left\|\psi_0\right\|_{H^{s+2}}+
%\left\|\dot \psi_0\right\|_{H^{s}}<\epsilon_{sr}\ , 
%\end{equation}
%then the corresponding solution fulfills
%\begin{equation}
%  \label{stima.sol}
%\left\| (\psi(t),\dot\psi(t))\right\|_{s}\leq
%C\epsilon\ \text{for}\ \left|t\right|\leq \frac{c}{\epsilon^{r}}
%\end{equation}
%\end{theorem}

%\subsection{From (\ref{EK3}) to Nonlinear Schr\"odinger}\label{madmad}

\subsubsection{Madelung transform}
For $\lambda\in \mathbb{R}_{+}$, 
we define the change of variable (\emph{Madelung transform})
\begin{equation*}\label{mad}\tag{$\mathcal{M}$}
%{\bf Madelung}}
 \psi:= \mathcal{M}_{\psi}(\rho,\phi):=\sqrt{\mathtt{m}+\rho} e^{{\rm i}\lambda \phi}\,,
 \quad \bar{\psi}:=\mathcal{M}_{ \bar\psi}(\rho,\phi):=\sqrt{\mathtt{m}+\rho}e^{-\ii\lambda\phi}\,.
  \end{equation*}
  Notice that the inverse map has the form
    \begin{equation}\label{madelunginv}
    \begin{aligned}
  \mathtt{m}+\rho&=\mathcal{M}_\rho^{-1}(\psi,\bar\psi):=|\psi|^{2}\,,
  %\qquad 
 \\
  \phi&=\mathcal{M}_\phi^{-1}(\psi,\bar\psi):=\frac{1}{\lambda}
  \arctan\left(\frac{-\ii(\psi-\bar{\psi})}{\psi+\bar{\psi}}\right)\,.
  \end{aligned}
  \end{equation}
  In the following lemma we state a well-posedness result for the Madelung transform.
  \begin{lemma}\label{federico1}
  Define $\kappa=(4\lambda^{2})^{-1}$ and $\hbar:=\lambda^{-1}=2\sqrt{\kappa}$.
%  \begin{equation*}%\label{alberobello}
%  \frac{1}{4\lambda^{2}}=\kappa\,,\qquad \hbar:= \frac{1}{\lambda}=2\sqrt{\kappa}\,.
%  \end{equation*}
  Then the following holds.
  
  \noindent
  $(i)$ Let $s>\frac{d}{2}$ and 
  \[
  \delta:=\frac{1}{\mathtt{m}}\|\rho\|_{s}+\frac{1}{\sqrt{\kappa}}\|\Pi_0^{\bot}\phi\|_{s} \,,
  \qquad \sigma:=\Pi_0\phi\,.
  \]
  There is $C=C(s)>1$ such that, if $C(s)\delta\leq1$, then the function 
  $\psi$ in \eqref{mad} satisfies
  \begin{equation*}%\label{stimamand}
  \|\psi-\sqrt{\mathtt{m}}e^{\ii\lambda\sigma}\|_{s}\leq 2\sqrt{\mathtt{m}}\delta\,.
  \end{equation*}
  
  \noindent
  $(ii)$ Define
  \[
  \delta':=\inf_{\sigma \in \T}\|\psi-\sqrt{\mathtt{m}}e^{\ii \sigma}\|_{s}\,.
  \]
  There is $C'=C'(s)>1$ such that, if $C'(s) \delta'(\sqrt{\mathtt{m}})^{-1}\leq 1$, then the functions $\rho,$
\begin{equation*}%\label{stimamand2}
\frac{1}{\mathtt{m}}\|\rho\|_{s}
+\frac{1}{\sqrt{\kappa}}\|\Pi_0^{\bot}\phi\|_{s}\leq 8\frac{1}{\sqrt{\mathtt{m}}}\delta'\,.
\end{equation*}

  \end{lemma}
  \begin{proof}
 see Lemma  2.1 in \cite{FeIaMu}.
  \end{proof}
%  \begin{remark}
%Notice that the map in \eqref{mad} is well-defined with  
%smooth inverse
%  \begin{equation}\label{madelunginv}
%  \mathtt{m}+\rho=|\psi|^{2}\,,\qquad \phi=\frac{1}{\lambda}
%  \arctan\left(\frac{-\ii(\psi-\bar{\psi})}{\psi+\bar{\psi}}\right)\,,
%  \end{equation}
%if (for instance) $\rho>0$ and $\lambda|\phi|\leq \pi/4$.
%  \end{remark}
  We now rewrite equation \eqref{EK3} in the variable $(\psi,\bar{\psi})$.
  \begin{lemma}\label{gatto23}
  Let $(\rho,\phi)\in H^s_0(\T^d)\times H^s(\T^d)$ be a solution of \eqref{EK3} defined over a time interval $[0,T]$, $T>0$,
  such that
  \begin{equation*}%\label{hawaii20}
  \sup_{t\in[0,T)}\Big(\frac{1}{\mathtt{m}}\|\rho(t,\cdot)\|_{s} +\frac{1}{\sqrt{\kappa}}
\| \Pi_0^{\bot}\phi(t,\cdot)\|_{s}    \Big)\leq \epsilon
  \end{equation*}
  for some $\epsilon>0$ small enough. Then 
  the function $\psi$ defined in \eqref{mad} solves
  \begin{equation}\label{papsipsi4}
    \begin{cases}
 \pa_{t}\psi=-\ii \big(-\frac{\hbar}{2}\Delta_g\psi+\frac{1}{\hbar}p(|\psi|^{2})\psi\big)\\
 \psi(0)= \sqrt{\mathtt{m}+\rho(0)}e^{\ii \phi(0)}\,.
 \end{cases}
 %=-\ii \pa_{\bar{\psi}}H(\psi,\bar{\psi})
 \end{equation}
  \end{lemma}
  \begin{proof}
  See Lemma 2.2 in \cite{FeIaMu}. 
  \end{proof}
    Notice that the \eqref{papsipsi4} is an Hamiltonian equation of the form
  \begin{equation}\label{papsipsi5}
 \pa_{t}\psi=-\ii \pa_{\bar \psi}\mathcal{H}(\psi,\bar{\psi})\,,\qquad
  \mathcal{H}(\psi,\bar{\psi})=
 \int_{\mathbb{T}^{d}}
 \big(\frac{\hbar}{2}|\nabla_g\psi|^{2}+\frac{1}{\hbar}P(|\psi|^{2})\big)dx\,,
 \end{equation}
 where $\pa_{\bar{\psi}}=(\pa_{\Re \psi}+\ii \pa_{\Im \psi})/2$.
  The Poisson bracket is defined by 
   \begin{equation}\label{Poisson}
\{\mathcal{H}, \mathcal{G}\}:=-
\ii \int_{\T^d} \pa_{\psi}\mathcal{H}\pa_{\bar{\psi}}\mathcal{G}
-  \pa_{\bar{\psi}}\mathcal{H}\pa_{{\psi}}\mathcal{G} dx\,.
\end{equation}

\subsubsection{Elimination of the zero mode}\label{eliminazero} 
%In this section we shall symmetrize the quadratic Hamiltonian in \eqref{sincity}.
%%Unfortunately for $j=0$
%%%Despite the fact that the linear system
%%%\[
%%%\pa_{t}\vect{z}{\bar{z}}=-\ii E\mathcal{L}\vect{z}{\bar{z}}\,,\qquad E:=\sm{1}{0}{0}{-1}\,,
%%%\]
%%% is diagonalizable for every mode $j\not=0$, we note that for $j=0$ 
%% the matrix in \eqref{linearizz} is not diagonalizable. For this reason 
% It is convenient to reduce the equation \eqref{papsipsi10} to 
% a system for a set of variable which do not include the zero mode. 
% Here we follow an idea introduced for the first time in \cite{Faouplane}.
 We introduce the set of variables 
\begin{equation}\label{faou}
\begin{cases}
 \psi_0= \alpha e^{-\ii \theta} & \alpha \in [0,+\infty)\,,\, \theta \in \T \\
 \psi_j=  z_j e^{-\ii \theta} & j\in \Z^d \setminus \{0\}\,,
\end{cases}
\end{equation} 
which are the polar coordinates for $j=0$ and a phase translation for $j\not=0$. 
Rewriting \eqref{papsipsi5} in Fourier coordinates one has 
\begin{equation*}
\ii\pa_t \psi_j = \pa_{\bar{\psi_j}}\mathcal{H}
(\psi,\bar \psi)\,, \quad j\in \Z^d\,,
\end{equation*}
where $\mathcal{H}$ is defined in \eqref{papsipsi5}. We define also the zero mean variable 
\begin{equation}\label{zeta}
z:= \sum_{j\in \Z^d \setminus\{ 0\} } z_j e^{\ii j\cdot x}\,.
\end{equation}
By \eqref{faou} and \eqref{zeta} one has 
\begin{equation}\label{faouinv}
\psi= ( \alpha + z) e^{\ii\theta}\,,
\end{equation}
and it is easy to prove that the quantity
\[
 \mathtt{m}:= \sum_{j\in \Z^d} |  \psi_j|^2= \alpha^2 +  \sum_{j \in\Z^d \setminus \{ 0 \}} | z_j|^2
\]
 is a constant of motion for \eqref{papsipsi4}. Using \eqref{faou}, 
 one can completely recover the variable $\alpha$ 
 in terms of $\{ z_j\}_{j\in \Z^d \setminus \{0\}}$ as 
\begin{equation*}%\label{def:alpha}
\alpha= \sqrt{ \mathtt{m}- \sum_{j \in \Z^d \setminus \{ 0 \}} |  z_j|^2}\,.
\end{equation*}
Note also  that the $(\rho,\phi)$ variables in \eqref{madelunginv} 
do not depend on the angular variable $\theta$ 
defined above. This implies that system \eqref{EK3} is 
completely described by the complex variable $z$.
On the other hand, using 
\[ 
%\frac{\pa \mathcal{H}}{\pa {\bar  \psi_j}}
\pa_{\bar{\psi_j}}\mathcal{H}(\psi e^{\ii \theta},\bar {\psi e^{\ii \theta}})
= 
%\frac{\pa \mathcal{H}}{\pa {\bar \psi_j}}
\pa_{\bar{\psi_j}}\mathcal{H}(\psi,\bar \psi)e^{\ii \theta}\,,
\] 
one obtains
\begin{equation} \label{equationnew}
\begin{cases}
\ii \pa_t \alpha+\pa_t\theta \alpha = \Pi_0\left( p( |\alpha+z|^2) (\alpha + z) \right)  
\\
\ii \pa_t  z_j + \pa_t\theta  z_j=   \frac{\pa \mathcal{H}}{\pa \bar \psi_j}(\alpha+ z ,\alpha +\bar z)\,.
\end{cases}
\end{equation}
Taking the real part of the first equation in \eqref{equationnew} we obtain 
\begin{equation} \label{deteta}
 \pa_t\theta= \frac{1}{\alpha}  
 \Pi_0\left( \frac{1}{\hbar}p( |\alpha+z|^2)\Re  (\alpha + z) \right) 
 = \frac{1}{2\alpha} 
 %\frac{\pa \tilde{\mathcal{H}}}{\pa \alpha } 
 \pa_{\bar{\alpha}}\mathcal{H}(\alpha, z, \bar z)\,,
\end{equation}
where  
\begin{equation*} 
\tilde{\mathcal{H}}(\alpha, z, \bar z):= 
\frac{\hbar}{2}\int_{\T^d} (- \Delta_g)z\cdot\bar{z}{\rm d}x
%| \nabla z|^2\, {\rm d} x
+ \frac{1}{\hbar}\int_{\T^d} G(| \alpha+z|^2)\, {\rm d}x\,.
\end{equation*}
By \eqref{deteta}, \eqref{equationnew} and using that  
\[ 
%\frac{\pa \mathcal{H}}{\pa \bar \psi_j}
\pa_{\bar{\psi_j}}\mathcal{H}(\alpha+ z ,\alpha +\bar z)
=
%\frac{\pa \tilde{\mathcal{H}}}{\pa \bar{z}_j } 
\pa_{\bar{z_{j}}}\tilde{\mathcal{H}}(\alpha, z, \bar z)\,,
\] 
one obtains 
\begin{equation} \label{fresca}
\begin{aligned}
\ii \pa_t  z_j =  & 
%\frac{\pa \tilde{\mathcal{H}}}{\pa \bar{z}_j } 
\pa_{\bar{z_j}}\tilde{\mathcal{H}}(\alpha, z, \bar z)
- \frac{z_j}{2\alpha}
%\frac{\pa \tilde{\mathcal{H}}}{\pa \alpha } 
\pa_{\alpha}\tilde{\mathcal{H}}(\alpha, z, \bar z)=
 %\frac{\pa \mathcal{K}_{\mathtt{m}}}{\pa \bar{z}_j } 
 \pa_{\bar{z_j}}\mathcal{K}_{\mathtt{m}}( z, \bar z)\,, \quad j\not=0\,, 
\end{aligned}
\end{equation}
where 
\begin{equation}\label{ostia1}
\mathcal{K}_{\mathtt{m}}(z,\bar z):= 
\tilde{\mathcal{H}}(\alpha,z,\bar z)_{|\alpha=\sqrt{\mathtt{m}- \sum_{j\not=0} | z_j|^2}}\,.
\end{equation}
We resume the above discussion in the following lemma.
\begin{lemma}\label{federico2}
The following holds. 

 $(i)$ Let $s>\frac{d}{2}$ and 
  \[
  \delta:=\frac{1}{\mathtt{m}}\|\rho\|_{s}+\frac{1}{\sqrt{\kappa}}\|\Pi_0^{\bot}\phi\|_{s} \,,
  \quad \theta:=\Pi_0\phi\,.
  \]
  There is $C=C(s)>1$ such that, if $C(s)\delta\leq1$, then the function $z$ in \eqref{zeta} satisfies
  \begin{equation*}%\label{stimamandzeta}
  \|z\|_{s}\leq 2\sqrt{\mathtt{m}}\delta\,.
  \end{equation*}
  
  \noindent
  $(ii)$ Define
  \[
  \delta':=\|z\|_{s}\,.
  \]
  There is $C'=C'(s)>1$ such that, if $C'(s) \delta'(\sqrt{\mathtt{m}})^{-1}\leq 1$, then the functions $\rho,$
\begin{equation*}%\label{stimamand2zeta}
\frac{1}{\mathtt{m}}\|\rho\|_{s}+\frac{1}{\sqrt{\kappa}}\|\Pi_0^{\bot}\phi\|_{s}
\leq 16\frac{1}{\sqrt{\mathtt{m}}}\delta'\,.
\end{equation*}
 
 \noindent $(iii)$ Let $(\rho,\phi)\in H^s_0(\T^d)\times H^s(\T^d)$ 
 be a solution of \eqref{EK3} defined over a time interval $[0,T]$, $T>0$,
  such that
  \begin{equation*}%\label{hawaii}
  \sup_{t\in[0,T)}\Big(\frac{1}{\mathtt{m}}\|\rho(t,\cdot)\|_{s} +\frac{1}{\sqrt{\kappa}}
\| \Pi_0^{\bot}\phi(t,\cdot)\|_{s}    \Big)\leq \epsilon
  \end{equation*}
  for some $\epsilon>0$ small enough. Then the  function $z\in H^s_0(\T^d)$ 
  defined in \eqref{zeta} solves \eqref{fresca}.

\end{lemma}
\begin{proof}
See Lemma 2.4 in \cite{FeIaMu}. 
%We note that 
%\begin{equation}
%\| z\|_{H^s}= \| \Pi_0^{\bot} \psi\|_{H^s}\leq \|\psi - \sqrt{m}e^{\ii \theta}\|_{H_s} \stackrel{\eqref{stimamand}}\leq2\sqrt{\mathtt{m}}\delta\,,
%\end{equation} 
%which proves \eqref{stimamandzeta}. 
%In order to prove \eqref{stimamand2zeta} we note that 
%\begin{equation*}
%\begin{aligned}
%\inf_{\sigma\in \T} \| \psi-\sqrt{\mathtt{m}}e^{\ii\sigma}\|_{H^s}&\leq 
%\| \psi- \sqrt{\mathtt{m}} e^{\ii \theta}\|_{H^s}
%=\| \alpha-\sqrt{\mathtt{m}}+ z\|_{H^s} 
%\\
%&\leq \sqrt{\mathtt{m}- \| z\|_{L^2}^2}-\sqrt{\mathtt{m}}+ \| z\|_{H^s}\leq 2\delta'\,,
%\end{aligned}
%\end{equation*}
%so that the  \eqref{stimamand2zeta} follows by \eqref{stimamand2}. 
%The point $(iii)$ follows by \eqref{equationnew} and \eqref{deteta}.
\end{proof}
\begin{remark}
Using \eqref{madelunginv} and \eqref{faouinv} one can study the system \eqref{EK3} 
near the equilibrium point $(\rho,\phi)=(0,0)$ by studying the complex hamiltonian system 
\begin{equation}\label{zetaequation}
\ii \partial_t z = \pa_{\bar{z}}\mathcal{K}_{\mathtt{m}}(z,\bar{z})
%\frac{\pa \mathcal{K}_{\mathtt{m}}}{\pa \bar z}(z,\bar z)
\end{equation}
near the equilibrium $z=0$, where $\mathcal{K}_{\mathtt{m}}(z,\bar{z})$
is the Hamiltonian in \eqref{ostia1}. 
Note also that the natural phase-space for \eqref{zetaequation} 
is the complex Sobolev space $H_0^s(\T^d)$, $s\in \R$, 
of complex Sobolev functions with zero mean.
\end{remark}
By Lemma \ref{federico2}, one has that Theorem \ref{main-idrodinamica} will be deduced by the following Proposition
\begin{proposition}\label{long time idro shrodinger}
Let $\bar g \in {\cal G}$. There exists a set of zero measure ${\cal B}^{(res)}\subset {\cal B}$, s.t. if
$\beta \in {\cal B} \setminus {\cal B}^{(res)}$ and $g = \beta \bar g$ then, $\forall r\geq 2$ there
exist $s_r$ and $\forall s>s_r$ $\exists \epsilon_{rs},c,C$ with the
following property. For any initial datum $z_0 \in H^s_0(\T^d)$ satisfying 
$$
\| z_0 \|_s \leq \epsilon
$$
there exists a unique solution $t \mapsto z(t)$ of the equation \eqref{fresca} satisfying the bound 
\[
\| z(t) \|_s \leq C \epsilon\,, 
\qquad \forall \, |t| \leq c\epsilon^{-r}\,. %\frac{c}{\epsilon^r}\,.  
\]
\end{proposition}
The rest of this section is dedicated to the proof of the latter Proposition.

\subsubsection{Taylor expansion of the Hamiltonian} % $\mathcal{K}_{\mathtt{m}}$} 
In this section we shall use the notations introduced in sections \ref{sezione spazio delle fasi}, \ref{functions}. The only difference is that, since we shall restrict to the space of zero average functions, in all the definitions given in sections \ref{sezione spazio delle fasi}, \ref{functions},  one has to replace $\Z^d$ by $\Z^d \setminus \{ 0 \}$ and ${\cal Z}^d$ by ${\cal Z}^d_0 := (\Z^d \setminus \{ 0 \}) \times \{ +, - \}$. In order to study the stability of $z=0$ for \eqref{zetaequation} it is useful to 
expand $\mathcal{K}_{\mathtt{m}}$ at $z=0$. 
We have
\begin{equation}\label{espansione}
\begin{aligned}
\mathcal{K}_{\mathtt{m}}(z,\bar z)&= 
\frac{\hbar}{2} \int_{\T^d} (- \Delta_g)z\cdot\bar{z}\, {\rm d} x 
+ \frac{1}{\hbar} \int_{\T^d} P\Big(\big| \sqrt{\mathtt{m}- \sum_{j\not=0} | z_j|^2}+ z\big|^2\Big)\, {\rm d}x 
\\&
=(2\pi)^d \frac{P(\mathtt{m})}{\hbar}
+\mathcal{K}_{\mathtt{m}}^{(2)}(z,\bar z)
+ \sum_{r= 3}^{N-1} \mathcal{K}_{\mathtt{m}}^{(r)}(z,\bar z)+R^{(N)}(z,\bar z)\,,
\end{aligned}
\end{equation}
where
\begin{equation*}
\mathcal{K}_{\mathtt{m}}^{(2)}(z,\bar z)=
\frac12 \int_{\T^d} \frac{\hbar}{2} (- \Delta_g) z\cdot\bar{z}\, {\rm d} x
+ \frac{p'(\mathtt{m})\mathtt{m}}{\hbar}\int_{\T^d} \frac12( z+\bar z)^2\, {\rm d}x\,,
\end{equation*}
for any $r=3, \cdots, N-1$, 
$\mathcal{K}_{{\rm m}}^{(r)}(z,\bar z)$ 
 is an homogeneous multilinear Hamiltonian function of degree $r$ of the form 
 \begin{equation*}
 \mathcal{K}_{\mathtt{m}}^{(r)}(z,\bar z)=
 \sum_{\substack{\sigma\in\{-1,1\}^r,\ j\in(\Z^d\setminus\{0\})^r
 \\ 
\sum_{i=1}^r\sigma_i j_i=0}}
(\mathcal{K}_{\mathtt{m}}^{(r)})_{\sigma,j}
z_{j_1}^{\sigma_1}\cdots z_{j_r}^{\sigma_r}\,,
\qquad 
|(\mathcal{K}_{\mathtt{m}}^{(r)})_{\sigma,j}|\lesssim_{r}1\,,
\end{equation*}
and
\begin{equation*}%\label{R8}
\| X_{R^{(N)}}(z)\|_{s}\lesssim_s \| z\|_{H^s}^{N-1}\,,
\qquad
\forall\, z\in B_{1}(
H_0^{s}(\mathbb{T}^{d}) \,.
\end{equation*}
This implies that ${\cal K}_m^{(r)}$ is in the class ${\cal P}_r$. 
%where we used the notation $A\lesssim B$ to denote 
%$A\le C B$ where $C$ is a positive constant
%depending on  parameters fixed once for all, for instance $d$
% and $s$.
% We will emphasize by writing $\lesssim_{q}$
% when the constant $C$ depends on some other parameter $q$.
%\end{subsection}
The vector field of the Hamiltonian in \eqref{espansione} has the form
\begin{equation*} %\label{linear}
\begin{aligned}
\partial_t\begin{bmatrix} z \\ \bar z\end{bmatrix}= 
\begin{bmatrix}-\ii \pa_{\bar{z}}  \mathcal{K}_{\mathtt{m}}\\ 
\ii\pa_{ z} \mathcal{K}_{\mathtt{m}}\end{bmatrix}
&= 
-\ii \begin{pmatrix} \frac{ \hbar \Delta_g}{2} 
+ \frac{\mathtt{m} p'(\mathtt{m})}{\hbar}&  
\frac{\mathtt{m}p'(\mathtt{m})}{\hbar}\\ 
-\frac{\mathtt{m} p'(\mathtt{m})}{\hbar}
 &\frac{\hbar \Delta_g}{2} 
 - \frac{\mathtt{m} p'(\mathtt{m})}{\hbar}\end{pmatrix}
 \begin{bmatrix} z \\ \bar z \end{bmatrix}
 \\&+
 \sum_{r=3}^{N-1}\begin{bmatrix} -\ii\pa_{\bar{z}}  \mathcal{K}_{\mathtt{m}}^{(r)}\\ 
\ii \pa_{ z} \mathcal{K}_{\mathtt{m}}^{(r)}\end{bmatrix}+
\begin{bmatrix} -\ii\pa_{\bar{z}}  R^{(N)}\\ 
\ii\pa_{ z} R^{(N)}\end{bmatrix}\,.
\end{aligned}
 \end{equation*}
 Let us now introduce the 
  $2\times2$ matrix of operators
 \[
 \mathcal{C}:=\frac{1}{\sqrt{2\omega(D)
A(D,\mathtt{m})}}
 \left(
 \begin{matrix}
A(D,\mathtt{m}) & 
 -\tfrac{1}{2}\mathtt{m}p'(\mathtt{m})
 \\
 -\tfrac{1}{2}\mathtt{m}p'(\mathtt{m}) & A(D,\mathtt{m}) \end{matrix}
 \right)\,,
 \]
 with 
  \[
A(D,\mathtt{m}):= \omega(D)
 +\tfrac{\hbar}{2} (- \Delta_g) +\tfrac{1}{2}\mathtt{m}p'(\mathtt{m})\,,
 \]
 and where $\omega(D)$ is the
 Fourier multiplier  with symbol 
 \begin{equation}\label{simboOmega}
 \begin{aligned}
 & \sqrt{ \frac{\hbar^2}{4} |j|_{g}^4+ \mathtt{m}p'(\mathtt{m}) |j|_{g}^2}  = \frac{\hbar}{2}\omega_j  \,, \quad j \in \Z^d \setminus \{ 0 \} \\
& \omega_j :=\sqrt{|j_g|^4 + \delta |j_g|^2}, \qquad  \delta := \frac{4\mathtt{m}p'(\mathtt{m})}{\hbar^2}\,. 
 \end{aligned}
 \end{equation}
 Notice that, by using \eqref{elliptic}, the matrix $\mathcal{C}$ is bounded, 
 invertible and symplectic, with estimates
 \begin{equation*}%\label{stimeC}
 \|\mathcal{C}^{\pm1}\|_{\mathcal{L}{(H^s_0\times H^s_0,\,\, H^s_0\times H^s_0)}}\leq1+\sqrt{k}\beta,\quad \beta:=\frac{\mathtt{m}p'(\mathtt{m})}{k}.
 \end{equation*}
 Consider the change of variables 
 \begin{equation*}%\label{def:WWW}
\begin{bmatrix} w \\ \bar w \end{bmatrix} := 
\mathcal{C}^{-1} \begin{bmatrix} z \\ \bar z \end{bmatrix}\,.
\end{equation*}
then the Hamiltonian \eqref{espansione} reads
\begin{equation*}%\label{HamKK}
\begin{aligned}
&\widetilde{\mathcal{K}}_{\mathtt{m}}=
\widetilde{\mathcal{K}}^{(2)}_{\mathtt{m}}+ \sum_{k = 3}^{N - 1} \widetilde{\cal K}_m^{(r)} + \tilde{R}_N
\\&
\widetilde{\mathcal{K}}^{(2)}_{\mathtt{m}}(w,\bar{w}):=
\mathcal{K}^{(2)}_{\mathtt{m}}\Big(\mathcal{C}
\begin{bmatrix} w \\ \bar w \end{bmatrix}\Big):=\frac{1}{2}
\int_{\mathbb{T}^{d}}\omega(D)w\cdot\bar{w}{\rm d}x\,,
\\&
\widetilde{\mathcal{K}}^{(i)}_{\mathtt{m}} \in {\cal P}_i \quad i = 3, \ldots, N - 1\,, \\
& \| X_{\widetilde R_N}(w) \|_s \lesssim_s \| w \|_s^{N - 1}\,, \quad \forall \| w \|_s \ll 1\,. 
\end{aligned}
\end{equation*}
From the latter properties, one deduces that the perturbation 
\[
P= \sum_{k = 3}^{N - 1} \widetilde{\cal K}_m^{(r)} + \tilde{R}_N\,, 
\]
is in the class $\cP$ of Def. \ref{funzP}. 

The verification of (F.3) goes exactly as in the case of the
Schr\"odinger equation, since also in this case $\omega_j = |j|_g^2 + O(1)$. 
The asymptotic condition (F.1) is also
trivially fulfilled with $\beta=2$.
The main point is to verify the nonresonance conditions (F.2) and
(NR.1), (NR.2). 
%This will occupy the rest of this subsection. 
This will be done in the next subsection.
\subsubsection{Non-resonance conditions for (\ref{EK3})}\label{sec:nonresQHD}
According to the section \ref{section beam} on the Beam equation, we fix the metric $\bar g \in {\cal G}$ and we consider $g = \beta\, \bar g$, $\beta_1 \leq \beta \leq \beta_2$. We shall verify the non-resonance conditions on the frequencies $\omega_j$ in \eqref{simboOmega}. By the property \eqref{j g bar g}, 
\begin{equation}\label{omega j alpha QHD}
\begin{aligned}
\omega_j  & = \sqrt{|j|_g^4 + \delta |j|_g^2} = \sqrt{\beta^4 |j|_{\bar g}^4 + \beta^2 \delta |j|_{\bar g}^2} = \beta^2 \Omega_j\,, \\
\Omega_j & :=  |j|_{\bar g} \sqrt{|j|_{\bar g}^2 + \frac{\delta}{\beta^2}}\,, \quad j \in \Z^d \setminus \{ 0 \}\,. 
\end{aligned}
\end{equation}
Since $\beta_2 \geq \beta \geq \beta_1 > 0$, 
$$
\Big|\sum_{i = 1}^r \sigma_i \omega_{j_i} \Big| =  \beta^2 \Big| \sum_{i = 1}^r \sigma_i \Omega_{j_i}\Big| \geq \beta_1^2  \Big| \sum_{i = 1}^r \sigma_i \Omega_{j_i}\Big|
$$
one can verify non resonance conditions on $\Omega_j$. Since the map 
\begin{equation}\label{mappa beta gamma QHD}
(\beta_1, \beta_2) \to (\zeta_1, \zeta_2) := \Big(\frac{\delta}{\beta_2^2}, \frac{\delta}{\beta_1^2} \Big)\,, \qquad  \beta \mapsto \zeta := \frac{\delta}{\beta^2}
\end{equation}
is an analytic diffeomorphism, we can introduce $\zeta = \delta / \beta^2$ as parameter in order to tune the resonances. Hence we verify non resonance conditions on the frequencies 
\begin{equation}\label{Omega j QHD}
\Omega_j \equiv \Omega_j (\zeta) =|j|_{\bar g} \sqrt{|j|_{\bar g}^2 + \zeta}, \quad j \in \Z^d \setminus \{ 0 \}\,. 
\end{equation}

\begin{lemma}
  \label{determinante idrodinamico} %[(Essentially Lemma 9 of \cite{Bam09})]
Assume that the metric $\bar g  \in {\cal G } $. For any $K\leq N$, consider $K$ indexes $\bji 1,...,\bji K$ with
$\left|\bji 1\right|_g< \ldots < \left|\bji K\right|_g\leq N$; and 
consider the determinant
\begin{equation*}
D:=\left|
\begin{matrix}
\Omega_{\bji 1}& \Omega_{\bji 2}&.& .&.&\Omega_{\bji K} 
\\
\partial_\zeta \Omega_{j_1} & \partial_\zeta \Omega_{j_2}& .& .&.& \partial_\zeta \Omega_{j_K}
\\
.& .& .& .& .&.
\\
.& .& .& .& .&.
\\
\partial_\zeta^{K - 1} \Omega_{j_1}& \partial_\zeta^{K - 1} \Omega_{j_2}& .& .&.& \partial_\zeta^{K - 1} \Omega_{j_K}
\end{matrix}
\right|
\end{equation*}
One has
\begin{equation*}
%\label{v.401}
D
\geq
\frac{C}{N^{\eta K^2}}
 \end{equation*}
for some constant $\eta \equiv \eta_d > 0$ large enough, depending only on the dimension $d$. 
\end{lemma}
\begin{proof}
The dispersion relation is slightly different w.r. to the one of the Beam equation, hence in this proof we just highlight the small differences w.r. to Lemma \ref{determinante}. For any $i = 1, \ldots, K$, for any $n = 0, \ldots, K - 1$, one computes 
$$
\partial_\zeta^n \Omega_{j_i}(\zeta) = C_n |j_i|_{\bar g} (|j_i|_{\bar g}^2 + \zeta)^{\frac12 - n}
$$
for some constant $C_n \neq 0$. This implies that 
$$
D \geq C \prod_{i = 1}^K \Big(|j_i|_{\bar g} \sqrt{|j_i|_{\bar g}^2 + \zeta} \Big) |{\rm det}(A)|
$$
where the matrix $A$ is defined as 
$$
A = \begin{pmatrix}
1& 1 &.& .&.&1
\\
x_1&x_2&.& .&.&x_K 
\\
.& .& .& .& .&.
\\
.& .& .& .& .&.
\\
x_1^{K  - 1}& x_2^{K - 1}&.& .&.&x_K^{K - 1} 
\end{pmatrix}
$$
where 
$$
x_i := \frac{1}{|j_i|_{\bar g}^2 + \zeta}, \quad i = 1, \ldots, K\,. 
$$
This is a Van der Monde determinant. Thus we have 
$$
\begin{aligned}
|{\rm det}(A)|  & = \prod_{1 \leq i < \ell \leq K} |x_i - x_\ell| = \prod_{1 \leq i < \ell \leq K} \Big| \frac{1}{|j_i|_{\bar g}^2 + \zeta} - \frac{1}{|j_\ell|_{\bar g}^2 + \zeta} \Big| \\
& \geq \prod_{1 \leq i < \ell \leq K}  \frac{||j_i|_{\bar g}^2 - |j_\ell|_{\bar g}^2|}{||j_i|_{\bar g}^2 + \zeta| ||j_\ell|_{\bar g}^2 + \zeta|} 
%\\& 
\stackrel{Lemma \,\,\ref{dista}}{\geq}  C N^{- K^2(\tau_* + 4)}
\end{aligned}
$$
which implies the thesis.
\end{proof}
Exploiting this Lemma, and following step by step the proof of Lemma
12 of \cite{Bam09} one gets

\begin{lemma}
\label{v.001 idrodinamico}
Let $\bar g \in {\cal G}$. Then for any $r$ there
exists $\tau_r$ with the following property: for any positive $\gamma$
small enough there exists a set $I_\gamma\subset (\zeta_1, \zeta_2)$ such
that $\forall \zeta\in I_\gamma$ one has that for any $N\geq 1$ and any
set $J_1,...,J_r$ with $|J_l|\leq N$ $\forall l$, one has
\begin{equation*}
%\label{v.002}
\sum_{l=1}^r\sigma_l\Omega_{j_l}\not =0\qquad \Longrightarrow\qquad
\left|\sum_{l=1}^r\sigma_l\Omega_{j_l}\right|
\geq\frac{\gamma}{N^{\tau}}\ .
\end{equation*}
Moreover one has
\begin{equation*}
%\label{v.003}
\left|[\zeta_1,\zeta_2]\setminus I_\gamma\right|\leq C \gamma^{1/r}\ .
\end{equation*} 
\end{lemma}
%The latter lemma, together with \eqref{omega j alpha QHD} and the diffeomorphism property \eqref{mappa beta gamma QHD}  
By recalling the diffeomorphism \eqref{mappa beta gamma QHD}, one has that the set 
$$
{\cal I}_\gamma := \{ \beta \in [\beta_1,  \beta_2] : \delta/\beta^2 \in I_\gamma \}
$$
satisfies the estimate 
$$
|(\beta_1, \beta_2) \setminus {\cal I}_\gamma| \lesssim \gamma^{\frac1r}
$$
Now, if we take $\beta \in {\cal I}_\gamma$ and if $\sum_{i=1}^r\sigma_i\omega_{j_i}\not =0$, one has that (recall \eqref{omega j alpha QHD})
$$
\begin{aligned}
\Big|\sum_{i = 1}^r \sigma_i \omega_{j_i} \Big| & =  \beta^2 \Big| \sum_{i = 1}^r \sigma_i \Omega_{j_i}\Big| \stackrel{\beta_1 \leq \beta \leq \beta_2}\geq \beta_1^2  \Big| \sum_{i = 1}^r \sigma_i \Omega_{j_i}\Big|  \\
& \geq \frac{\beta_1^2 \gamma}{N^\tau}\,. 
\end{aligned}
$$
By the above result,
one has  that, if
$$
\beta \in\bigcup_{\gamma>0}{\cal I}_\gamma\ ,
$$
then (NR.2) holds and furthermore $\bigcup_{\gamma>0}{\cal I}_\gamma$ has
full measure. Hence the claimed statement follows by defining ${\cal B}^{(res)} := {\cal B} \setminus \Big( \bigcup_{\gamma>0}{\cal I}_\gamma \Big)$. %\noindent
%{\it End of the proof of Lemma \ref{non.beam}} By the above result,
%one has that for any $g\in\cG\setminus\cG^{(res)}$ one has that, if
%$$
%m\in\bigcup_{\gamma>0}I_\gamma\ ,
%$$
%then (NR.2) holds and furthermore $I^g:=\bigcup_{\gamma>0}I_\gamma$ has
%full measure. Here we put in evidence the fact that the set depends on
%$g$. To conclude the proof of the theorem, just define
%$\cS^{(good)}:=\bigcup_{g\in\cG\setminus\cG^{(res)} }(g,I^g) $ and
%$\cS^{(res)}:=\left(\cG\times[m_0,\Delta]\right)\setminus \cS^{(good)}$\qed

%Therefore system \eqref{linear} becomes
%\begin{equation}\label{linear22}
%\pa_{t}w=-\ii \omega(D)w-\ii \pa_{\bar{w}}\widetilde{\mathcal{K}}^{(3)}_{\mathtt{m}}(w,\bar{w})
%-\ii \pa_{\bar{w}}\widetilde{\mathcal{K}}^{(\geq4)}_{\mathtt{m}}(w,\bar{w})\,.
%\end{equation}

\subsection{Stability of plane waves in NLS}\label{plane}

Consider the NLS
\begin{equation}
  \label{PNLS}
\im\psi_t=-\Delta_g \psi+f(|\psi|^2)\psi\ ,
  \end{equation}
with $f\in C^\infty(\R,\R)$, $f(0) = 0$ and $g = \beta \bar g$, $\bar g\in\cG$ and $\beta \in (\beta_1, \beta_2) \subset (0, + \infty)$. 
(recall the Definition \ref{parametrimetrica}). 
The equation \eqref{PNLS}  admits solutions of the form
\begin{equation}
  \label{NLSP.1}
\psi_{*,m}(x,t)=ae^{\im (m\cdot x-\nu t)}\ ,\quad m\in\Z^d
  \end{equation}
with $\nu=|m|_g^2+f(a^2)$ and $a > 0$. In order to state the next stability theorem, we need that a suitable condition between $f'(a^2)$ and the metric $g$ is satisfied. For this reason, we slightly modify the definition of ${\cal G}_0$ in \ref{parametrimetrica}. We then re-define ${\cal G}_0$ in the following way: fix $K > 0$, we define
\begin{equation}\label{def cal G0 new}
\mathcal{G}_{0}:=\left\{
\left(g_{ij}\right)_{i\leq j}\in \R^{\frac{d (d + 1)}{2}}\; : \; \inf_{x \neq 0} \frac{g(x, x)}{|x|^2} > K
\right\}
\end{equation}
The definition of the admissible set ${\cal G}$ is then the same in which one replace this new set ${\cal G}_0$ with its hold definition. 
The main theorem of this section is the following. 
\begin{theorem}
  \label{main.plane}
Assume that $0 < \beta_1 < \beta_2$, $\bar g \in {\cal G}$, $2f'(a^2)< \beta_1^2 K^2$, $f'(a^2) \neq 0$ (where $K > 0$ is the constant appearing in \eqref{def cal G0 new}). 
Then  there exists a set of zero measure 
${\cal B}^{(res)} \subset {\cal B} := (\beta_1, \beta_2)$, such that for
$\beta\in {\cal B} \setminus {\cal B}^{(res)} $ for $g = \beta \bar g$, then, for any $ r\geq3$, there
exist $s_r>0$ such that the following holds.
For any $ s>s_r$ and any  $ m\in\Z^d$ there exist constants  $ \epsilon_{rsm},c,C$ 
such that
%with the following property: 
if the initial datum $\psi_0$ for \eqref{PNLS} fulfills
\begin{equation}
  \label{plane.dat}
\left\|\psi_0\right\|_{L^2}=a\sqrt{|\T^d|_g}\ , 
\quad
\epsilon:=\left\|\psi_0-\psi_{*,m}(.,0)  \right\|_{H^s}<\epsilon_{srm}\ , 
\end{equation}
then the corresponding solution fulfills
\begin{equation}
  \label{stima.sol.plane}
\left\| \psi(t)-\psi_{*,m}(.,t)\right\|_{s}\leq
C\epsilon\,,\quad \forall\;  \left|t\right|\leq c\epsilon^{-r}\,. 
%\frac{c}{\epsilon^{r}}
\end{equation}
\end{theorem}
Arguing as in the proof of Corollary \ref{remark misura piena beam}, one can show also in this case the following  
\begin{corollary}\label{remark misura piena plane waves}
Let $0 < \beta_1 < \beta_2$. There exists a zero measure set ${\cal G}^{(res)}_{\beta_1, \beta_2} \subseteq {\cal G}_0(\beta_1, \beta_2)$, where ${\cal G}_0(\beta_1, \beta_2) := \big\{ g \in {\cal G}_0 : \beta_1 \leq \| g \|_2 \leq \beta_2 \big\}$, such that for any $g \in {\cal G}_0(\beta_1, \beta_2) \setminus {\cal G}^{(res)}_{\beta_1, \beta_2}$ the statements of theorem \ref{main.plane} hold. 
\end{corollary}

The rest of this subsection is devoted to sketch the proof of Theorem \ref{main.plane}, which follows
exactly the proof of the corresponding theorem in \cite{FGL13} except
that in the case of nonresonant tori one has to substitute the
nonresonant condition by \cite{FGL13} with our nonresonance 
and
structure conditions (see Hypotheses \ref{hypo1}, \ref{hypo2}).

We start by reducing the problem to a problem of stability of the origin of
a system of the form \eqref{h.abs}.

First it is easy to see that introducing the new variables $\varphi$
by
\begin{equation*}
%  \label{togli}
\varphi(x,t)=e^{-\im m\cdot x}e^{-\im
  t\left|m\right|^2}\psi(x+2mt,t)\ ,
\end{equation*}
then $\varphi$ still fulfills \eqref{PNLS}, but $\psi_{*,m}(x,t)$ is
changed to $a e^{-\im\nu t}$ with $\nu= f(a^2)$.

The idea of \cite{FGL13} is to exploit that $\varphi(x)=a$  appears as
an elliptic equilibrium of the reduced Hamiltonian system obtained  applying
Marsden Weinstein procedure to \eqref{PNLS} in order to reduce the
Gauge symmetry. We recall that according to Marsden Weinstein
procedure (following \cite{FGL13}), when one has a system invariant under a one parameter 
symmetry group, then there exists an integral of motion (the $L^2$
norm in this case), and the effective dynamics occurs in the quotient of the level
surface of the integral of motion with respect to the group
action. This is the same procedure exploited in section \ref{sec:QHD} for the QHD system. The effective system has a Hamiltonian which is obtained by
restricting the Hamiltonian to the level surface. Such a Hamiltonian
is invariant under the symmetry group associated to the integral of motion.

More precisely, consider the zero mean variable
$$
z(x):=\frac{1}{|\T^d|_g^{1/2}}\sum_{j\in\Z^d\setminus
  \left\{0\right\}} z_j e^{\im j\cdot x}\ ,
$$
and the substitution
\begin{equation}
  \label{varphi}
\varphi(x)=e^{\im \theta} (\sqrt{a^2-\left|\T^d\right|_g\left\|
  z\right\|^2_{L^2}}+z(x) ) 
  \end{equation}
where $\theta \in \T$ is a parameter along the orbit of the Gauge group,
Notice that $\varphi$ belongs to the level surface
$\norma{\varphi}_{L^2}=a\sqrt{\left|\T^d\right|_g} $ and $z(x)$ is the
new free variable. In this case it also turns out that this is a
canonical variable(as it can be verified by the theory of \cite{Bam13}). Thus the Hamiltonian for the reduced system turns
out to be
\begin{equation*}
 % \label{ham.PNLS}
H_a(z,\bar z)=\int_{\T^d} \left(\bar\varphi(-\Delta \varphi)+F(|\varphi|^2)\right)dx \ ,
  \end{equation*}
with $\varphi$ given by \eqref{varphi}. The explicit form of the
Hamiltonian and its expansion were computed in \cite{FGL13} who showed
that all the terms of the Taylor expansion of $H_a$ have zero momentum
and that all the nonlinear terms are bounded, so, with our language,
the nonlinear part is of class $\cP$. Considering the quadratic part,
\cite{FGL13} showed that there exists a linear transformation
preserving $H^s$ norms and the zero momentum condition, such that the
quadratic part takes the form \eqref{H0} with
\begin{equation}
  \label{PNLS.5}
\omega_j=\sqrt{|j|_g^4-f'(a^2)|j|_g^2}\ .
  \end{equation}
The system is now suitable for the application of Theorem
\ref{main.abs}. We do not give the details, since the verification of the
nonresonance and structural assumptions are done exactly in the same
way as in the previous cases. Indeed one can prove 
the nonresonance conditions on the frequencies \eqref{PNLS.5}
reasoning as done in section \ref{sec:nonresQHD}.

\appendix

\section{A technical lemma}\label{lemma}

{\bf In this section by $\ell^2_s$ we mean $\ell^2_s(\Z^d;\C)$.  }

\begin{lemma}
  \label{prod}
Let 
\[
X:\underbrace{\ell^2_s\times...\times
  \ell^2_s}_{r-\text{times}}\to\ell^2_s\,,
  \]
   be a symmetric $r$-linear
 $X(u^{(1)},...,u^{(r)})=\left(X_j(u^{(1)},...,u^{(r)})\right)_{j\in\Z^d}$ 
 with the property that there exist 
 $\sigma_0,\sigma_1,...,\sigma_r$, with
$\sigma_l\in\left\{-1,1\right\}$
such that
\begin{equation}
  \label{lX}
  X_j(u^{(1)},...,u^{(r)})=\sum_{\begin{subarray}{c}j_1,...,j_r\in\Z^d
      \\  \sigma_0j+\sum_{l=1}^{r}\sigma_lj_l=0 \end{subarray}
  }X_{j,j_1,...,j_r} u_{j_1}^{(1)}....u_{j_r}^{(r)}\ ,
\end{equation}
and $X_{j,j_1,...,j_r}$ 
completely symmetric with respect to any
permutation of the indexes $j,j_1,...,j_r$ fulfilling
\[
\sup_{j,j_1,...,j_r\in\Z^{d}}\left|X_{ j,j_1,...,j_r}\right|<\infty\ .
\]
Then, for any $s>s_0>d/2$
there exists a constant $C_{s,r}>0$ such that 
one has
\begin{equation*}\label{tamemap}
\begin{aligned}
 \left\|X(u^{(1)},...,u^{(r)})\right\|_s 
 &\leq C_{s,r}
  \sup_{j,j_1,...,j_r\in\Z^d}\left|X_{ j,j_1,...,j_r}\right|\times
  \\&
  \times
\sum_{l=1}^{r}
\norma{u^{(1)}}_{s_0}....\norma{u^{(l-1)}}_{s_0} \norma{u^{(l)}}_{s}
\norma{u^{(l+1)}}_{s_0}...\norma{u^{(r)}}_{s_0}\ .
\end{aligned}
\end{equation*}
\end{lemma}
\proof One has
\begin{align*}
  \left\|X(u^{(1)},...,u^{(r)})\right\|_s=
  \sum_{j}\langle j\rangle^{2s}
  \left| 
  \sum_{\begin{subarray}{c}j_1,...,j_r\in\Z^d
      \\  \sigma_0j+\sum_{l=1}^{r}\sigma_lj_l=0 
      \end{subarray}}
      X_{j,j_1,...,j_r} u_{j_1}^{(1)}\cdots u_{j_r}^{(r)}  \right|^{2}
\\
  \leq 
 \sup_{j,j_1,...,j_r}
 \left|X_{ j,j_1,...,j_r}\right|^2 
 \sum_{j}\langle j\rangle^{2s}
\left( \sum_{\begin{subarray}{c}j_1,\ldots,j_r\in\Z^d
      \\  \sigma_0j+\sum_{l=1}^{r}\sigma_lj_l=0 \end{subarray}
  } |u_{j_1}^{(1)}\cdots u_{j_r}^{(r)} | \right)^{2}
\end{align*}
To fix ideas consider first the case $\sigma_l=1$ $\forall l$, then
the bracket is the $j$-th Fourier  coefficient of the function
$v(x)=u^{(1)}(x)\cdots u^{(r)}(x)$, 
with
\[
u^{(l)}(x)=\sum_{j\in\Z^d}\frac{1}{|\T^d|^{1/2}} u^{(l)}_{j}e^{\im
  j\cdot x}\ ,
  \]
for which it is well known that 
\begin{equation}
  \label{tamemap.1}
  \nonumber \left\|v\right\|_s \leq C_{s,r}
\sum_{l=1}^{r}
\norma{u^{(1)}}_{s_0}....\norma{u^{(l-1)}}_{s_0} \norma{u^{(l)}}_{s}
\norma{u^{(l+1)}}_{s_0}...\norma{u^{(r)}}_{s_0}\ .
\end{equation}
then the thesis immediately follows. To deal with the case of
different signs every time one has $\sigma_l=-1$ one simply
substitutes $\overline{u^{(l)}}$ to $u^{(l)} $. This concludes the proof.  
\qed

%\addcontentsline{toc}{chapter}{Bibliography} 

%\bibliography{../biblio}

\def\cprime{$'$}

\end{document}